\documentclass[reqno]{amsart}
\usepackage{graphicx}
\usepackage[dvips]{epsfig}
\usepackage{bm}
\usepackage{mathrsfs}
\usepackage[breakable]{tcolorbox}
\usepackage{geometry}
\allowdisplaybreaks

 \makeatletter
\@namedef{subjclassname@2020}{%
  \textup{2020} Mathematics Subject Classification}
\makeatother

\newtheorem{theorem}{Theorem}[section]
\newtheorem{lemma}{Lemma}[section]
\newtheorem{proposition}[lemma]{Proposition}%[setcion]
\newtheorem{definition}[lemma]{Definition}%[section]
\newtheorem{remark}[lemma]{Remark}%[section]
\newtheorem{problem}{{\bf{Problem}}}[section]

\numberwithin{figure}{section}

\numberwithin{equation}{section}

\allowdisplaybreaks
\begin{document}

\title[TWO-PHASE FREE BOUNDARY PROBLEMS FOR EULER FLOWS]{A Two-Phase Free Boundary Problem for Axisymmetric Subsonic Euler Flows with Contact Discontinuities}

\author{Hyangdong Park}
\address{Department of Mathematical Sciences, Korea Advanced Institute of Science and Technology (KAIST), 291 Daehak-ro, Yuseong-gu, Daejeon 34141, Republic of Korea}
\email{hyangdong@kaist.ac.kr}



\date{\today}

\begin{abstract} 
We study a free boundary problem for the three-dimensional steady compressible Euler equations in an infinitely long circular cylinder. The free boundary is a contact discontinuity separating an axisymmetric rotational subsonic flow and an axisymmetric potential subsonic flow, neither of which is prescribed a priori. The pressure continuity condition couples two unknown Euler phases through an unknown interface, leading to a genuinely two-phase free boundary problem. Using a Helmholtz decomposition, we reformulate the pressure continuity condition as nonlinear boundary conditions for the Helmholtz variables and develop a coupled iteration framework that simultaneously determines the free boundary and the two flow fields. Uniform estimates independent of the truncation length allow us to pass to the infinite-length limit and construct global solutions. We further establish the downstream asymptotic behavior of solutions, showing that the contact discontinuity becomes asymptotically cylindrical and that the radial velocity vanishes at infinity. To the best of our knowledge, this provides the first existence result and asymptotic characterization for a genuinely two-phase free boundary problem involving a contact discontinuity between an unknown rotational subsonic flow and an unknown potential subsonic flow in a three-dimensional infinitely long cylinder.
\end{abstract}

\maketitle

%%%%%%%%%%%%%%%%%%%%%%%%%%%%%%%%%%%

\tableofcontents

\section{Introduction}
Free boundary problems generated by characteristic discontinuities constitute one of the central themes in the mathematical theory of compressible flows. Among them, contact discontinuities for the steady Euler equations present a particularly delicate analytical challenge. Across a contact discontinuity, the pressure remains continuous while the density, entropy, and tangential velocity may undergo jumps. Consequently, the location of the interface is not prescribed a priori and must be determined simultaneously with the surrounding flow fields through the Rankine–Hugoniot conditions.
The difficulty becomes substantially more pronounced in the subsonic regime. Unlike supersonic flows, the steady subsonic Euler system possesses a mixed hyperbolic–elliptic structure. Transport quantities such as entropy and angular momentum density propagate along streamlines, while the velocity field is determined through nonlinear elliptic equations. As a result, perturbations of a contact discontinuity influence both transport and elliptic components in a highly nonlocal manner. This interaction gives rise to a free boundary problem whose solvability remains largely open in multidimensional settings.
Most of the existing existence theories for subsonic contact discontinuities possess an important structural simplification. Either one side of the discontinuity is prescribed a priori, or both phases are restricted to irrotational configurations. In such settings, the interface acts essentially as a boundary condition for a single unknown phase. The resulting free boundary problem is therefore effectively one-phase in nature.
%%%%%

%%

The purpose of the present paper is to investigate a fundamentally different configuration. We consider a three-dimensional infinitely long circular cylinder and seek contact discontinuities separating an axisymmetric rotational subsonic Euler flow from an axisymmetric potential subsonic flow. Neither phase is prescribed a priori, and the free boundary together with the two adjacent flow states must be determined simultaneously.
The presence of a potential phase does not eliminate the nonlinear coupling inherent in the free boundary problem. Although the potential phase is governed solely by an elliptic subsystem, the rotational phase still contains both transport and elliptic components through the entropy and angular momentum density. The pressure continuity condition across the contact discontinuity therefore couples an unknown potential flow, an unknown rotational flow, and an unknown interface. Consequently, the contact discontinuity no longer acts as a boundary condition relative to a prescribed background state. Instead, it becomes a nonlinear transmission interface linking two distinct Euler phases through an unknown free boundary.
From the viewpoint of the Helmholtz decomposition, the rotational phase consists of a transport subsystem governing entropy and angular momentum density together with a nonlinear elliptic subsystem determining the velocity field, whereas the potential phase is governed by a nonlinear elliptic system. The pressure continuity condition therefore generates a nonlinear coupling among the transport variables, the elliptic variables, and the free boundary itself. This coupling prevents the two phases from being constructed independently and leads to a genuinely two-phase free boundary problem.
In particular, although one phase is irrotational, its state is not prescribed a priori and must be determined simultaneously with the rotational phase and the free boundary.

The main observation of this paper is that the pressure continuity condition can be reformulated as nonlinear boundary conditions for the Helmholtz variables. This reformulation reveals a transmission mechanism through which the potential phase, the rotational phase, and the free boundary interact. Motivated by this observation, we develop a coupled iteration framework in which the free boundary and the two Euler phases are constructed simultaneously while preserving the Rankine–Hugoniot conditions.
 A major difficulty arises from the fact that the contact discontinuity extends through an infinitely long cylinder. To overcome this issue, we introduce a family of free boundary problems posed in truncated cylinders and establish a priori estimates that are uniform with respect to the truncation length. These estimates allow us to pass to the infinite-length limit and construct a global contact discontinuity separating an axisymmetric potential subsonic flow and an axisymmetric rotational subsonic flow.
To the best of our knowledge, no previous existence theory appears to be available for
 a contact discontinuity separating an unknown rotational subsonic Euler flow and an unknown potential subsonic flow in a three-dimensional infinitely long cylinder.
 
 In addition to the existence theory, we also investigate the far-field behavior of the resulting solutions. Although the pressure continuity condition couples two unknown Euler phases through a free boundary, the flow exhibits a partial stabilization mechanism downstream. We prove that the contact discontinuity becomes asymptotically cylindrical and that the radial velocity vanishes at infinity. As a consequence, the radial momentum equation is asymptotically governed by the balance between the pressure gradient and the centrifugal force. These asymptotic properties provide a quantitative description of the downstream structure of the two-phase flow and reveal how the rotational effects persist in the far field.

%%%%%%%%%%%%%%%%%%%%%%%%%%%%%%%%%%%
The study of free boundary problems generated by contact discontinuities for the steady Euler equations has attracted considerable attention in recent years. For isentropic Euler flows, free boundary problems associated with contact discontinuities were investigated in \cite{bae2013stability}. For two-dimensional subsonic Euler flows with large vorticity and characteristic discontinuities, an existence theory was established in various nozzle geometries in \cite{MR3914482}. The analysis of rotational subsonic Euler flows with nonzero vorticity has also been investigated in \cite{bae2019contact,bae2019contact3D}, where a Helmholtz decomposition was employed to reformulate the Euler system and one-phase free boundary problems were studied.
The analysis of contact discontinuities has also been carried out in transonic and supersonic settings. Stability theories for transonic and supersonic contact discontinuities were established in \cite{HUANG20194337,Huang:2021aa}. 
Unlike the aforementioned works, the present paper treats a genuinely two-phase free boundary problem for rotational subsonic Euler flows in an infinitely long cylinder.

The main contributions of the paper can be summarized as follows.
\begin{itemize}

\item[(1)] We formulate contact discontinuities separating an axisymmetric rotational subsonic Euler flow and an axisymmetric potential subsonic flow as a two-phase free boundary problem.

\item[(2)] We reformulate the pressure continuity condition as nonlinear boundary conditions for the Helmholtz variables. This reformulation reveals the transmission mechanism coupling the free boundary, the transport subsystem, and the elliptic subsystem associated with the two Euler phases.

\item[(3)] We develop a coupled iteration framework that simultaneously determines the free boundary and the two Euler phases while preserving the Rankine–Hugoniot conditions throughout the iteration procedure.

\item[(4)] Using uniform a priori estimates independent of the truncation length, we construct global solutions in an infinitely long cylinder.

\item[(5)] We establish the far-field asymptotic structure of the two-phase flow. In particular, the contact discontinuity becomes asymptotically cylindrical, the radial velocity vanishes, and the radial momentum equation is asymptotically balanced by the pressure gradient and centrifugal force associated with the angular momentum density.\end{itemize}

%%%%%%%%%%%%%%%%%%%%%%%%%%%%%%%%%%%
The paper is organized as follows. In Section \ref{S-2}, we formulate the free boundary problem and state the main theorem. In Section \ref{S-3}, we reformulate the Euler system by means of the Helmholtz decomposition and derive an equivalent two-phase free boundary problem in terms of the Helmholtz variables. This reformulation converts the pressure continuity condition into nonlinear boundary conditions. % and provides the analytic framework for the coupled iteration scheme.
The core of the analysis is carried out in Sections \ref{S-4} and \ref{S-6}. In Section \ref{S-4}, we study free boundary problems in truncated cylinders and construct a coupled iteration map that simultaneously updates the free boundary and the two Euler phases. Uniform estimates independent of the truncation length are established at this stage. In Section \ref{S-6}, these estimates are used to pass to the infinite-length limit and obtain solutions in the original infinitely long cylinder.
Finally, Section \ref{S-7} is devoted to the proof of Theorem \ref{main=Thm-infty}. In Section \ref{sub-ex}, we verify that solutions of the reformulated problem generate weak solutions of the steady Euler system with a contact discontinuity. In Section \ref{sub-asym}, we establish the downstream asymptotic behavior of solutions, including the asymptotic cylindrical shape of the contact discontinuity and the decay of the radial velocity.

%%%%%%%%%%%%%%%%%%%%%%%%%%%%%%%%%%%%%%
\section{Problem and Main results}\label{S-2}
In $\mathbb{R}^n$, the steady compressible flows of ideal polytropic gases are governed by the steady Euler system:
\begin{equation}\label{Euler}
\left\{
\begin{split}
&\mbox{div}(\rho{\bf u})=0,\\
&\mbox{div}(\rho{\bf u}\otimes{\bf u})+\nabla p={\bf0},\\
&\mbox{div}(\rho{\bf u} B)=0,
\end{split}
\right.
\end{equation}
for the density $\rho$, velocity ${\bf u}=(u_1,\ldots,u_n)$, pressure $p$, and the Bernoulli invariant $B$ defined by
\begin{equation*}
B(\rho,{\bf u},p):=\frac{|{\bf u}|^2}{2}+\frac{\gamma p}{(\gamma-1)\rho}
\end{equation*}
for the adiabatic exponent $\gamma>1$.
As is well known, the type of system \eqref{Euler} depends on the type of flow.
If the flow is supersonic (i.e., $|{\bf u}|>\sqrt{\gamma p/\rho}$), then the system is hyperbolic. 
If the flow is subsonic (i.e., $|{\bf u}|<\sqrt{\gamma p/\rho}$), then the system is of mixed hyperbolic-elliptic type.

%%%%%%%%%%%%%%%%%%%%%%%%%%%%%%%%%%%
%%%%
Let $\Omega\subset\mathbb{R}^n$ be an open connected set, and let $\Gamma$ be a non-self-intersecting $C^1$ surface of codimension $1$ dividing $\Omega$ into two disjoint open subsets $\Omega^{\pm}$ such that $\Omega=\Omega^-\cup\Gamma\cup\Omega^+$.
%%%%%%%%%%%%%%%%%%%%%%%%%%%%%%%%%%%
\begin{definition}%[Weak solution to \eqref{Euler}]
We define ${\bf U}=(\rho,{\bf u},p)\in [L_{\rm loc}^{\infty}(\Omega)\cap C_{\rm loc}^1(\Omega^{\pm})\cap C_{\rm loc}^0(\Omega^{\pm}\cup\Gamma)]^{n+2}$ to be a weak solution to \eqref{Euler} in $\Omega$ if the following properties hold: For any test function $\xi\in C_0^\infty(\Omega)$ and $k=1,\dots,n$,
\begin{equation}\label{weak}
\int_{\Omega}\rho{\bf u}\cdot\nabla\xi d{\bf x}=\int_{\Omega}(\rho u_k{\bf u}+p{\bf e}_k)\cdot\nabla\xi d{\bf x}=\int_{\Omega}\rho{\bf u}B\cdot\nabla\xi d{\bf x}=0,
\end{equation}
where each ${\bf e}_k$ is the unit vector in the $x_k$-direction.
\end{definition}
%%%%%%%%%%%%%%%%%%%%%%%%%%%%%%%%%%%
%%

By integration by parts, ${\bf U}$ satisfies \eqref{weak} if and only if 
\begin{itemize}
\item[$(w_1)$] ${\bf U}$ is a classical solution to \eqref{Euler} in $\Omega^{\pm}$;
\item[$(w_2)$] ${\bf U}$ satisfies the Rankine-Hugoniot conditions
	\begin{itemize}
	\item[(RH1)] $[\rho{\bf u}\cdot{\bf n}]_{\Gamma}=0$, $[\rho({\bf u}\cdot{\bf n})^2+p]_{\Gamma}=0$, $[\rho{\bf u}\cdot{\bf n} B]_{\Gamma}=0$,
	\item[(RH2)] $\rho({\bf u}\cdot{\bf n})[{\bf u}\cdot{\bm\tau}_k]_{\Gamma}=0$ for all $k=1,\dots, n-1$, 
	\end{itemize}
	where $[\,\cdot\,]_{\Gamma}$ defined by  $[F({\bf x})]_{\Gamma}:=F({\bf x})_{\overline{\Omega^-}}-F({\bf x})_{\overline{\Omega^+}}$ for ${\bf x}\in\Gamma$, ${\bf n}$ is a unit normal vector field on $\Gamma$, and ${\bm \tau}_k$ $(k=1,\dots,n-1)$ are unit tangent vector fields on $\Gamma$ such that they are linearly independent at each point on $\Gamma$.
\end{itemize}
%%%%%%%%%%%%%%%%%%%%%%%%%%%%%%%%%%%
When $\rho>0$,  (RH2) is satisfied if either ${\bf u}\cdot{\bf n}=0$ on $\Gamma$ or $[{\bf u}\cdot{\bm\tau}_k]_{\Gamma}=0$ for all $k=1,\dots,n-1$.
If ${\bf u}\cdot {\bf n}=0$ on $\Gamma$ and $[p]_\Gamma=0$, then $\Gamma$ is a contact discontinuity. 
When ${\bf u}\cdot {\bf n}\neq 0$ and the corresponding Rankine--Hugoniot conditions are satisfied, the discontinuity is a shock.
In this paper, we study contact discontinuities. For shocks, one can refer to \cite{bae2011transonic,  chen2004steady, chen2007existence, chen2018mathematics} and the references cited therein.

\begin{definition}\label{def-cd}
We define ${\bf U}=(\rho,{\bf u},p)\in [L_{\rm loc}^{\infty}(\Omega)\cap C_{\rm loc}^1(\Omega^{\pm})\cap C_{\rm loc}^0(\Omega^{\pm}\cup\Gamma)]^{n+2}$ to be a weak solution to the Euler system \eqref{Euler} in $\Omega$ with a contact discontinuity $\Gamma$ provided that the following properties hold:
\begin{itemize}
\item[(i)] $\Gamma$ is a non-self-intersecting $C^1$-surface of co-dimension 1 dividing $\Omega$ into two open subsets $\Omega^{\pm}$ such that $\Omega=\Omega^+\cup\Gamma\cup\Omega^-$;
\item[(ii)] ${\bf U}$ is a classical solution to \eqref{Euler} in $\Omega^{\pm}$;
\item[(iii)] $\rho>0$ in $\overline{\Omega}$;
%\item[(iv)] $({\bf u}|_{\overline{\Omega^-}\cap\Gamma}-{\bf u}|_{\overline{\Omega^+}\cap\Gamma})({\bf x})\ne{\bf 0}$ holds for all ${\bf x}\in\Gamma$;
\item[(iv)] ${\bf u}\cdot{\bf n}|_{\overline{\Omega^-}\cap\Gamma}={\bf u}\cdot{\bf n}|_{\overline{\Omega^+}\cap\Gamma}=0$, where ${\bf n}$ is a unit normal vector field on $\Gamma$;
\item[(v)] $[p]_{\Gamma}=0$. 
\end{itemize}
\end{definition}
%%%%%%%%%%%%%%%%%%
%%%%%%%%%%%%%%%%%%%

%%%%%%%%%%%%%%%%%%%%%%
\begin{figure}[!h]
\centering
\includegraphics[scale=0.75]{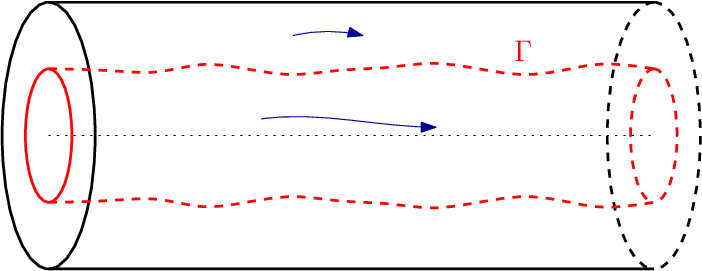}
\caption{Contact discontinuity $\Gamma$ in a 3-D circular cylinder}\label{fig0}
\end{figure}
%%%%%%%%%%%%%%%%%%%%%%%%%%%%%%%%%%%
%%%%%%%%%%%
We are concerned with the existence of contact discontinuities in a three-dimensional circular cylinder $\Omega$ of infinite length defined as follows:
\begin{equation*}
\Omega:=\left\{(x_1,x_2,x_3)\in\mathbb{R}^3:\,x_1>0,0\le\sqrt{x_2^2+x_3^2}<1\right\}.
\end{equation*}
Denote the entrance and the  wall of the nozzle $\Omega$ as follows:
\begin{equation*}
\Gamma_0:=\partial\Omega\cap\{x_1=0\},\quad\Gamma_{\rm w}:=\partial\Omega\cap\left\{\sqrt{x_2^2+x_3^2}=1\right\}.
%,\quad\Gamma_L:=\partial\Omega\cap\{x_1=L\}.
\end{equation*}
For a function $f:\mathbb{R}^{\ge0}\to(\frac{1}{4},\frac{3}{4})$ with $f(0)=\frac{1}{2}$, define $\Omega_f^-$ and $\Omega_f^+$ by 
\begin{equation*}
\Omega_f^-:=\Omega\cap\left\{0<\sqrt{x_2^2+x_3^2}<f(x_1)\right\}\mbox{ and }\Omega_f^+:=\Omega\cap\left\{f(x_1)<\sqrt{x_2^2+x_3^2}<1\right\}.
\end{equation*}
For each region of $\Omega_f^{\pm}$, we denote the entrance as follows:
\begin{equation*}
\begin{split}
&\Gamma_{0}^\pm:=\partial\Omega_f^\pm\cap\Gamma_0. %,\quad \Gamma_{\rm w}:=\partial\Omega_f^{+}\cap\Gamma_{\rm w}.
%,\quad \Gamma_{f,L}^{\pm}:=\partial\Omega_f^{\pm}\cap\Gamma_L.
\end{split}
\end{equation*}

%%%%%%%%%%t%%%%%%
%%%
To study flows in a circular cylinder $\Omega$, we use a standard cylindrical coordinate system $(x,r,\theta)$ given by the relation
\begin{equation*}
x_1=x,\quad x_2=r\cos\theta,\quad x_3=r\sin\theta,\quad\mbox{$r\in\mathbb{R}^+$,\quad$\theta\in\mathbb{T}$,}
\end{equation*}
where $\mathbb{T}$ is a one-dimensional torus with period $2\pi$. 
%In this paper, we study axisymmetric flows. Let us define axisymmetric functions as follows:
\begin{definition}
Let us define an orthonormal basis $\{{\bf e}_x,{\bf e}_r,{\bf e}_{\theta}\}$ by 
\begin{equation*}
{\bf e}_x:=(1,0,0),\quad {\bf e}_r:=(0,\cos\theta,\sin\theta),\quad {\bf e}_{\theta}:=(0,-\sin\theta,\cos\theta).
\end{equation*}
Let $D$ be an open subset of $\mathbb{R}^3$ and let ${\bf x}\in D$.
\begin{itemize}
\item[(i)] A function $f:D\to\mathbb{R}$ is axisymmetric if $f({\bf x})=\tilde{f}(x,r)$ for some function $\tilde{f}:\mathbb{R}^2\to\mathbb{R}$.
\item[(ii)] A vector-valued function ${\bf F}:D\to\mathbb{R}^3$ is axisymmetric if ${\bf F}=F_x{\bf e}_x+F_r{\bf e}_r+F_{\theta}{\bf e}_{\theta}$  for some axisymmetric functions $F_x$, $F_r$, and $F_\theta$.
\end{itemize}
\end{definition}

%%%%%%%%%%t%%%%%%%%%%%%%%%%%%%%%%%%%
%%%%%%%%%%%%%%%%%%%
%%%%%%%%%%t%%%%%%%%%%%%%%%%%%%%%%%%%
The goal of this work is to prove the existence of an axisymmetric subsonic solution to the Euler system \eqref{Euler} with a contact discontinuity separating a rotational phase, which may have nonzero vorticity and nonzero angular momentum density, from a potential phase in a three-dimensional infinitely long cylinder.
%The goal of this work is to prove the existence of an axisymmetric subsonic ($|{\bf u}|<\sqrt{\gamma p/\rho}$) solution to the Euler system \eqref{Euler} with non-zero vorticity $(\nabla\times{\bf u}\ne0$), non-zero angular momentum density $({\bf u}\cdot{\bf e}_{\theta}\ne0$), and a contact discontinuity in the sense of Definition \ref{def-cd} in a three-dimensional infinitely long cylinder $\Omega$. 
%%%%%%%%%%%%%%%%%%%%%%%%%%%%%%%%%%%%%%
%
%

We first consider a simple solution.
For distinct positive constants $u_0^{\pm}$ with $0<u_0^+<u_0^-$, define two velocity vectors ${\bf u}_{0}^{\pm}$ by 
\begin{equation*}
\mbox{${\bf u}_0^+:=(u_0^{+},0,0)$ and ${\bf u}_0^-:=(u_0^-,0,0)$.}
\end{equation*} For positive constants $\rho_0^{\pm}$ and $p_0$ with  
\begin{equation*}
u_0^{\pm}<\sqrt{\gamma p_0\slash \rho^{\pm}},
\end{equation*}
it is clear that a vector-valued function ${\bf U}_0$ defined by 
\begin{equation}\label{background}
{\bf U}_0({\bf x}):=\left\{\begin{split}
				(\rho_0^+,{\bf u}_0^+,p_0)\quad&\mbox{in}\,\, \Omega_{\frac{1}{2}}^+,\\
				(\rho_0^-,{\bf u}_0^-,p_0)\quad&\mbox{in}\,\, \Omega_{\frac{1}{2}}^-,
				\end{split}\right.
\end{equation}
is a subsonic solution to the Euler system \eqref{Euler} with a contact discontinuity on the surface $\mathfrak{C}_{\frac{1}{2}}:={\Omega}\cap\{\sqrt{x_2^2+x_3^2}=\frac{1}{2}\}$.
For this solution, the entropy $S_0$ and the Bernoulli invariant $B_0$ are piecewise constant functions with 
\begin{equation}\label{S0-def}
\begin{split}
&S_0:=\left\{\begin{split}
				\frac{p_0}{(\rho_0^+)^{\gamma}}=:S_0^+\quad&\mbox{in}\,\,\Omega_{\frac{1}{2}}^+,\\
				\frac{p_0}{(\rho_0^-)^{\gamma}}=:S_0^-\quad&\mbox{in}\,\, \Omega_{\frac{1}{2}}^-,
				\end{split}\right.\\
&B_0:=\left\{\begin{split}
				\frac{|{\bf u}_0^+|^2}{2}+\frac{\gamma p_0}{(\gamma-1)\rho_0^+}=:B_0^+\quad&\mbox{in}\,\, \Omega_{\frac{1}{2}}^+,\\
				\frac{|{\bf u}_0^-|^2}{2}+\frac{\gamma p_0}{(\gamma-1)\rho_0^-}=:B_0^-\quad&\mbox{in}\,\, \Omega_{\frac{1}{2}}^-.
				\end{split}\right.
\end{split}
\end{equation}

%%%%%%%%%%%%%%%%%%%
\begin{figure}[!h]
\centering
\includegraphics[scale=0.8]{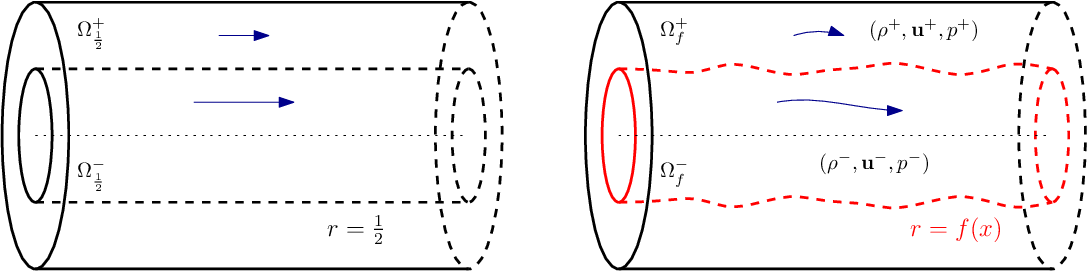}
\caption{Left: background state, right: Problem \ref{P-1-infty}}\label{fig1}
\end{figure}

%%%%%%%%%%%%%%%%%%%

We are concerned with the following problem.
%%%%%%%%%%%%%%%
\begin{problem}\label{P-1-infty} %%%%en도 pm으로 나눠야 될 듯
For given radial functions $w_{\rm en}^{\pm}:\Gamma_0^{\pm}\to\mathbb{R}$, $S_{\rm en}^{\pm}:\Gamma_0^{\pm}\to\mathbb{R}$, and  $v_{\rm en}^\pm:\Gamma_0^{\pm}\to\mathbb{R}$,
find a weak solution ${\bf U}:=(\rho^{\pm},{\bf u}^{\pm},p^{\pm})$ to the Euler system \eqref{Euler} in $\Omega$ with a contact discontinuity 
\begin{equation*}
\mathfrak{C}_f:={\Omega}\cap\{r=f(x)\}
\end{equation*}
in the sense of Definition \ref{def-cd} such that the following properties hold:
\begin{itemize}
\item[(i)] $\rho^{\pm}$, ${\bf u}^{\pm}$, and $p^{\pm}$ are axisymmetric;
\item[(ii)] $f(0)=\frac{1}{2}$;
\item[(iii)] {\rm (Positivity of the density)} $\rho^{\pm}>0$ in $\overline{\Omega_f^{\pm}}$;
\item[(iv)] {\rm (Bernoulli invariant)} $B(\rho^{\pm},{\bf u}^{\pm},p^{\pm})\equiv B_0^{\pm}$ in $\overline{\Omega_f^{\pm}}$ for $B_0^{\pm}$ given in \eqref{S0-def};
\item[(v)] {\rm (Subsonic speed)} $|{\bf u}^{\pm}|<c(\rho^{\pm},p^{\pm})$ in $\overline{\Omega_f^{\pm}}$ for the sound speed $c(\rho,p):=\sqrt{\frac{\gamma p}{\rho}}$;
\item[(vi)] {\rm (Slip boundary condition)} ${\bf u}^{+}\cdot{\bf e}_r=0$ on $\Gamma_{\rm w}$;% for a normal vector field ${\bf n}$ on $\Gamma_{\rm w}^{\pm}$;
\item[(vii)] {\rm (Entrance data)} ${\bf u}^{\pm}\cdot(-{\bf e}_r)=v_{\rm en}^{\pm}$, ${\bf u}^{\pm}\cdot{\bf e}_{\theta}=w_{\rm en}^{\pm}$, and $\frac{p^{\pm}}{(\rho^{\pm})^{\gamma}}=S_{\rm en}^{\pm}$ on $\Gamma_0^{\pm}$;
%\item[(vii)] ${\bf u}^{\pm}\cdot(0,1)=0$ on $\Gamma_{f,L}^{\pm}$;% for a normal vector field ${\bf n}$ on $\Gamma_{f,L}^{\pm}$;
\item[(viii)] {\rm (Rankine-Hugoniot conditions)} ${\bf u}^{\pm}\cdot{\bf n}_f^{\pm}=0$ and $p^+=p^-$ on $\mathfrak{C}_{f}$ for outward unit normal vector fields ${\bf n}_f^{\pm}$ on ${\Omega_f^{\pm}}\cap \mathfrak{C}_f$.
\end{itemize}

\end{problem}
%%%%%%%%%%%%%%%%%%%%%Uniqueness는??

The main results of the paper are the following:

\begin{theorem}\label{main=Thm-infty}
Fix $\epsilon\in(0,\frac{1}{10})$ and $\alpha\in(0,1)$. 
Suppose that $w_{\rm en}^{\pm}:\Gamma_0^{\pm}\to\mathbb{R}$, $S_{\rm en}^{\pm}:\Gamma_0^{\pm}\to\mathbb{R}$, and  $v_{\rm en}^\pm:\Gamma_0^{\pm}\to\mathbb{R}$ are radial functions and the following conditions hold:
\begin{eqnarray}
%&&(S_{\rm en}^-,w_{\rm en}^-)=(S_0^-,0)\quad\mbox{on}\quad\Gamma_0^-\cap\left\{r<\epsilon\right\},\\
&\label{po-con}&S_{\rm en}^+=S_0^+,\quad w_{\rm en}^+\equiv 0,\\
&\label{ven-con}&v_{\rm en}^+(r)=0\quad\mbox{on}\quad\Gamma_0^+\backslash\left\{\frac{1}{2}+\epsilon\le r\le 1-\epsilon\right\}=:\Gamma_{0,\epsilon}^+,\\
&\label{ven-con-2}& v_{\rm en}^-(r)=0\quad\mbox{on}\quad\Gamma_0^-\backslash\left\{r\le \frac{1}{2}-\epsilon\right\}=:\Gamma_{0,\epsilon}^-.
\end{eqnarray}
Define $\sigma^{\pm}$ and $\sigma$ by 
\begin{equation*}
\begin{split}
&\sigma^+:=\|v_{\rm en}^+\|_{1,\alpha,\Gamma_0^{+}},\\
&\sigma^-:=\|S_{\rm en}^--S_0^-\|_{2,\alpha,\Gamma_0^{-}}+\|w_{\rm en}^-{\bf e}_{\theta}\|_{2,\alpha,\Gamma_0^{-}}+\|v_{\rm en}^-\|_{1,\alpha,\Gamma_0^{-}},\\
&\sigma:=\sigma^++\sigma^-.
\end{split}
\end{equation*}
%Under the same assumptions \eqref{ven-con}-\eqref{ven-con-2} as in Problem \ref{P-1-infty},
\begin{itemize}
\item[(a)] There exist small constants $\varepsilon_0>0$, independent of $u_0^+$, depending only on ($\rho_0^{\pm}$, $u_0^{-}$, $p_0$, $\gamma$, $\alpha$) and $\sigma_0>0$ depending only on ($\rho_0^{\pm}$, $u_0^{\pm}$, $p_0$, $\gamma$, $\alpha$) so that if 
\begin{equation*}
u_0^+\le \varepsilon_0\mbox{ and }
\sigma\le\sigma_0,
\end{equation*}
then Problem \ref{P-1-infty} has a solution $(f,\rho^{\pm},{\bf u}^{\pm},p^{\pm})$ that satisfies 
\begin{equation}\label{u-rho-p-est-infty}
%\begin{split}
\|f-\frac{1}{2}\|_{2,\alpha,(0,\infty)}+\|\rho^{\pm}-\rho_0^{\pm}\|_{1,\alpha,\Omega_f^{\pm}}+\|{\bf u}^{\pm}-{\bf u}_0^{\pm}\|_{1,\alpha,\Omega^{\pm}_f}+\|p^{\pm}-p_0\|_{1,\alpha,\Omega_f^{\pm}}\\
\le C\sigma
%\end{split}
\end{equation}
for a positive constant $C$ depending only on $\rho_0^{\pm}$, $u_0^{\pm}$, $p_0$, $\gamma$, and $\alpha$.
\item[(b)] There exists a constant $\sigma_0^{\ast}\in(0,\sigma_0]$ depending only on ($\rho_0^{\pm}$, $u_0^{\pm}$, $p_0$, $\gamma$, $\alpha$) so that if $\sigma\le\sigma_0^{\ast}$, then 
\begin{equation*}
\begin{split}
&\lim_{L\to\infty}\|f'(x)\|_{C^1(\{x\ge L\})}=0,\\
&\lim_{L\to\infty}\|{\bf u}^{\pm}\cdot{\bf e}_r(x,\cdot)\|_{C^1(\overline{\Omega_{f}^{\pm}\cap\{x>L\}})}=0,\\
&\lim_{L\to\infty}\|\partial_rp^{\pm}(x,\cdot)-\frac{\rho^{\pm}({\bf u}^{\pm}\cdot{\bf e}_{\theta})^2}{r}(x,\cdot)\|_{C^0(\overline{\Omega_{f}^{\pm}\cap\{x>L\}})}=0.
\end{split}
\end{equation*}
\end{itemize}
\end{theorem}
%%%%%%%%%%%%%
\begin{remark}
The smallness assumption on $u_0^+$ and \eqref{po-con} allows us to close the coupled iteration scheme while preserving the genuine two-phase coupling mechanism. %This is the reason why we consider a potential flow in $\Omega_f^+$ and we assume the condition \eqref{po-con}.
%The smallness condition on $u_0^+$ is required to close the coupled iteration scheme.
%This is the reason why we consider a potential flow in $\Omega_f^+$ and the conditions in \eqref{po-con}. 
%The conditions in \eqref{ven-con}-\eqref{ven-con-2} are compatibility conditions to apply the reflection method to resolve the corner singularity issues. 
\end{remark}
%%%%%%%%%%%%%%%%%%%%%%%%%%%%%%%%%%%%%%

%%%%%%%%%%%%%%%%%%%%%%%%%%%%%%%%%%%%%%

%\newpage
\section{Reformulation via Helmholtz decomposition}\label{S-3}
In this section, we reformulate Problem \ref{P-1-infty} and Theorem \ref{main=Thm-infty} by using the Helmholtz decomposition of the velocity vector field.
%%
%Throughout the paper, $\partial_{x_i}$ is abbreviated to $\partial_i$.

Define a function $\varrho$ by 
\begin{equation}\label{varrho-def}
\varrho(\zeta,{\bf q}):=\left(\frac{(\gamma-1)\left(B_0-\frac{1}{2}|{\bf q}|^2\right)}{\gamma\zeta}\right)^{\frac{1}{\gamma-1}}
\end{equation}
for $\zeta\in\mathbb{R}$, ${\bf q}\in\mathbb{R}^3$. 
We express $\rho$, ${\bf u}$, and $p$ as 
\begin{equation}\label{express}
 \rho=\varrho(S,\nabla\varphi+\nabla\times{\bf V}),\quad 
{\bf u}=\nabla\varphi+\nabla\times{\bf V},\quad p=S\varrho^{\gamma}(S,\nabla\varphi+\nabla\times{\bf V})
\end{equation}
for axisymmetric functions
\begin{equation*}
\varphi({\bf x})=\varphi(x,r),\quad {\bf V}({\bf x})=h_1(x,r){\bf e}_x+h_2(x,r){\bf e}_r+\psi(x,r){\bf e}_{\theta}.
\end{equation*}
A direct computation yields 
\begin{equation*}
{\bf u}
=\,\nabla\varphi+\nabla\times(\psi{\bf e}_{\theta})+\frac{\Lambda}{r}{\bf e}_{\theta}
=:\,{\bf q}(\nabla\varphi,\nabla\times(\psi{\bf e}_{\theta}),\frac{\Lambda}{r}{\bf e}_{\theta}),
\end{equation*}
where $\Lambda$ is an angular momentum density defined by 
\begin{equation}\label{def-lambda}
\Lambda(x,r):=r{\bf u}\cdot{\bf e}_{\theta}(x,r)=r(\partial_xh_2-\partial_rh_1)(x,r).
\end{equation}
Note that the decomposition of ${\bf u}$ is not unique and we will find $\Lambda$, not $h_1$ and $h_2$. 
With \eqref{express}-\eqref{def-lambda}, the Euler system \eqref{Euler} can be rewritten as a system for $(S,\Lambda,\varphi,\psi)$:
\begin{equation}\label{re-system-HD}
\left\{\begin{split}
&\mbox{div}(\varrho(S,{\bf q}){\bf q})=0,\\
&-\Delta(\psi{\bf e}_{\theta})=G(S,\Lambda,\partial_rS,\partial_r\Lambda,\nabla\varphi,{\bf t}){\bf e}_{\theta},\\
&\varrho(S,{\bf q}){\bf q}\cdot\nabla S=0,\\
&\varrho(S,{\bf q}){\bf q}\cdot\nabla \Lambda=0,
\end{split}\right.
\end{equation}
with ${\bf q}:={\bf q}(\nabla\varphi,\nabla\times(\psi{\bf e}_{\theta}),\frac{\Lambda}{r}{\bf e}_{\theta})$ and ${\bf t}:={\bf q}-\nabla\varphi$,
for $G$ defined by 
\begin{equation*}
G(\zeta_1,\zeta_2,\zeta_3,\zeta_4,{\bf a},{\bf b}):=\frac{1}{({\bf a}+{\bf b})\cdot{\bf e}_x}\left(\frac{\zeta_3\varrho^{\gamma-1}(\zeta_1,{\bf a}+{\bf b})}{\gamma-1}+\frac{\zeta_2\zeta_4}{r^2}\right)
\end{equation*}
for $\zeta_1,\zeta_2,\zeta_3,\zeta_4\in\mathbb{R}$, ${\bf a}$, ${\bf b}\in\mathbb{R}^3$.
%%%%%
%%%%%%%%%%%%%%%%

%One can easily check that if the following two-phase free boundary problem holds, then (i), (ii), (iv), (vi), (vii), and (viii) in Problem \ref{P-1-infty} hold.
We solve the following two-phase free boundary problem:

\begin{problem}\label{H-free-infty}\label{Problem-HD}
Under the same assumptions as in Theorem \ref{main=Thm-infty},
find a function $f:\mathbb{R}^{\ge0}\to(\frac{1}{4},\frac{3}{4})$ with $f(0)=\frac{1}{2}$  and an axisymmetric solution $(S^{\pm},\Lambda^{\pm},\varphi^{\pm},\psi^{\pm})$ to the following free boundary problem:
\begin{equation}\label{3D-Free-prob}
\left\{\begin{split}
&\mbox{\rm div}(\varrho^{\pm}(S^{\pm},{\bf q}^{\pm}){\bf q}^{\pm})=0,\\
&-\Delta(\psi^{\pm}{\bf e}_{\theta})=G^{\pm}(S^{\pm},\Lambda^{\pm},\partial_rS^{\pm},\partial_r\Lambda^{\pm},\nabla\varphi^{\pm},{\bf t}^{\pm}){\bf e}_{\theta},\\
&\varrho^{\pm}(S^{\pm},{\bf q}^{\pm}){\bf q}^{\pm}\cdot\nabla S^{\pm}=0,\\
&\varrho^{\pm}(S^{\pm},{\bf q}^{\pm}){\bf q}^{\pm}\cdot\nabla \Lambda^{\pm}=0,
\end{split}\right.\quad\mbox{in}\,\,\Omega_{f}^{\pm}
\end{equation}
with the boundary conditions
\begin{equation}\label{3D-Free-bd}%%psi컨디션 체크
\left\{
\begin{split}
\partial_r\varphi^{\pm}=-v_{\rm en}^{\pm},\quad \partial_x\psi^{\pm}=0,\quad S^{\pm}=S_{\rm en}^{\pm},\quad\Lambda^{\pm}=rw_{\rm en}^{\pm}\quad&\mbox{on}\,\,\Gamma_{0}^{\pm},\\
\partial_r\varphi^+=0,\quad \partial_x\psi^+=0\quad&\mbox{on}\,\,\Gamma_{\rm w},\\
\nabla\varphi^{\pm}\cdot{\bf n}_f^{\pm}+(\nabla\times\psi^{\pm}{\bf e}_{\theta})\cdot{\bf n}_f^{\pm}={0}\quad&\mbox{on}\,\,\mathfrak{C}_{f},\\
S^+(\varrho^+)^{\gamma}(S^{+},{\bf q}^+)=S^-(\varrho^-)^{\gamma}(S^{-},{\bf q}^-) \quad&\mbox{on}\,\,\mathfrak{C}_{f}\\
%\varphi^{\pm}=\varphi^{\pm}_0(L,\cdot),\quad \partial_x\psi^{\pm}=0\quad&\mbox{on}\,\,\Gamma_{f,L}^{\pm},\\
\end{split}
\right.
\end{equation}
for ${\bf q}^{\pm}:={\bf q}(\nabla\varphi^{\pm},\nabla\times(\psi^{\pm}{\bf e}_{\theta}),\frac{\Lambda^{\pm}}{r}{\bf e}_{\theta})$, ${\bf t}^{\pm}:={\bf q}^{\pm}-\nabla\varphi^{\pm}$, and
\begin{equation*}
{\bf n}_f^{\pm}:=\frac{\pm f'(x_1){\bf e}_x\mp{\bf e}_r}{\sqrt{1+|f'(x_1)|^2}},
\end{equation*}
where $\varrho^{\pm}$, $G^{\pm}$, and $\varphi_0^{\pm}$ are defined by 
\begin{equation}\label{3D-varphi0}
\begin{split}
&\varrho^{\pm}(\zeta,{\bf q}):=\left(\frac{(\gamma-1)\left(B_0^{\pm}-\frac{1}{2}|{\bf q}|^2\right)}{\gamma\zeta}\right)^{\frac{1}{\gamma-1}},\\
&G^{\pm}(\zeta_1,\zeta_2,\zeta_3,\zeta_4,{\bf a},{\bf b}):=\frac{1}{({\bf a}+{\bf b})\cdot{\bf e}_x}\left(\frac{\zeta_3(\varrho^{\pm})^{\gamma-1}(\zeta_1,{\bf a}+{\bf b})}{\gamma-1}+\frac{\zeta_2\zeta_4}{r^2}\right),\\
&\varphi_0^{\pm}({\bf x})=u_0^{\pm}x_1%,\quad \varphi_0^{-}({\bf x}):=u_0^{-}x_1
\end{split}
\end{equation}
for $\zeta\in\mathbb{R}$, ${\bf q}\in\mathbb{R}^3$, $\zeta_1,\zeta_2,\zeta_3,\zeta_4\in\mathbb{R}$, ${\bf a}$, ${\bf b}\in\mathbb{R}^3$, and ${\bf x}=(x_1,x_2,x_3)\in\mathbb{R}^3$.\end{problem}

A direct computation shows that  if the boundary conditions \eqref{3D-Free-bd} for $(S^{\pm},\Lambda^{\pm}, \varphi^{\pm},\psi^{\pm})$ hold, then the boundary conditions (vi)-(viii) in Problem \ref{P-1-infty}  for $(\rho^{\pm},{\bf u}^{\pm},p^{\pm})$  hold.

\begin{theorem}\label{3D-Theorem-HD}
For any fixed $\alpha\in(0,1)$,
there exist small constants $\varepsilon_1>0$, independent of $u_0^+$, depending only on ($\rho_0^{\pm}$, $u_0^{-}$, $p_0$, $\gamma$, $\alpha$) and $\sigma_1>0$ depending only on ($\rho_0^{\pm}$, $u_0^{\pm}$, $p_0$, $\gamma$, $\alpha$) so that if 
\begin{equation*}
u_0^+\le\varepsilon_1\mbox{ and }
\sigma\le \sigma_1,
\end{equation*}
then Problem \ref{H-free-infty} has a  solution $(f,S^{\pm},\Lambda^{\pm},\varphi^{\pm},\psi^{\pm})$ that satisfies 
\begin{equation}\label{3D-est-Thm-HD}
\begin{split}
\|f-\frac{1}{2}\|_{2,\alpha,(0,+\infty)}+\|S^{\pm}-S_0^{\pm}\|_{1,\alpha,\Omega^{\pm}_f}+\|\frac{\Lambda^{\pm}}{r}{\bf e}_{\theta}\|_{1,\alpha,\Omega^{\pm}_f}&\\
+\|\varphi^{\pm}-\varphi_0^{\pm}\|_{2,\alpha,\Omega^{\pm}_f}+\|\psi^{\pm}{\bf e}_{\theta}\|_{2,\alpha,\Omega^{\pm}_f}&
\le C\sigma
\end{split}
\end{equation}
for a positive constant $C$ depending only on $\rho_0^{\pm}$, $u_0^{\pm}$, $p_0$, $\gamma$, and $\alpha$. 
%In the above, $S_0^{\pm}$ and $\varphi_0^{\pm}$ are defined by \eqref{S0-def} and \eqref{3D-varphi0}.
\end{theorem}
Hereafter, a constant is said to be chosen depending only on the data if it is chosen depending only on $\rho_0^{\pm}$, $u_0^{\pm}$, $p_0$, $\gamma$, and $\alpha$.

The remainder of the paper is devoted to the proof of Theorem \ref{3D-Theorem-HD}.
Once Theorem \ref{3D-Theorem-HD} is established, Theorem \ref{main=Thm-infty} follows immediately.

%%%%%%%%%%%%%%%%%%%%%%%%%%%%%%%%%%%%
%%%%%%%%%%%%%%%%%%%%%%%%%%%%%%%%%%%%

\section{Free boundary problems in truncated cylinders}\label{S-4}

Unlike the one-phase free boundary problem studied in \cite{bae2019contact3D}, the present free boundary problem couples two unknown flow fields through the pressure continuity condition across the contact discontinuity.
Since the pressure continuity condition involves the traces of both Euler states on the unknown interface, neither phase can be solved independently while keeping the interface fixed.
Consequently, the free boundary and the two Euler phases cannot be constructed sequentially.
Instead, they must be determined simultaneously within a coupled iteration framework.
The purpose of this section is to construct such an iteration map and establish its solvability.

\subsection{Problem and Proposition}
For a fixed constant $L>10$, define a cylinder $\Omega_L$ of length $L$ by
\begin{equation*}
\Omega_L:=\Omega\cap\{x<L\}.
\end{equation*}
We denote the wall and the exit of the nozzle $\Omega_L$ as follows:
\begin{equation*}
%\Gamma_0:=\partial\Omega\cap\{x_1=0\},\quad
\Gamma_{{\rm w},L}:=\partial\Omega_L\cap\{r=1\},
\quad\Gamma_L:=\partial\Omega_L\cap\{x=L\}.
\end{equation*}
For a function $f:\mathbb{R}^{\ge0}\to(\frac{1}{4},\frac{3}{4})$ with $f(0)=\frac{1}{2}$, define  $\Omega_{f,L}^+$, $\Omega_{f,L}^-$, and $\mathfrak{C}_{f,L}$ by 
\begin{equation*}
\begin{split}
&\Omega_{f,L}^+:=\Omega_L\cap\{f(x)<r<1\},\quad \Omega_{f,L}^-:=\Omega_L\cap\{r<f(x)\},\\
&\mathfrak{C}_{f,L}:={\Omega_L}\cap\{r=f(x)\}.
\end{split}
\end{equation*}
For each region of $\Omega_{f,L}^{\pm}$, we denote the entrance  and the exit 
of the nozzle as follows:
\begin{equation*}
\begin{split}
&\Gamma_{0}^\pm:=\partial\Omega_f^\pm\cap\Gamma_0,
\quad \Gamma_{f,L}^{\pm}:=\partial\Omega_{f,L}^{\pm}\cap\Gamma_L.
\end{split}
\end{equation*}
As defined in \eqref{ven-con}-\eqref{ven-con-2}, $\Gamma_{0,\epsilon}^{\pm}$ are defined by 
\begin{equation*}
\Gamma_{0,\epsilon}^+:=\Gamma_0^+\backslash\left\{\frac{1}{2}+\epsilon\le r\le 1-\epsilon\right\},\quad
\Gamma_{0,\epsilon}^-:=\Gamma_0^-\backslash\left\{r\le \frac{1}{2}-\epsilon\right\}.
\end{equation*}

\begin{problem}\label{H-free}
Under the same assumptions as in Theorem \ref{main=Thm-infty},
find a function $f:[0,L]\to(\frac{1}{4},\frac{3}{4})$ and an axisymmetric solution $(S^{\pm},\Lambda^{\pm},\varphi^{\pm},\psi^{\pm})$ to the following free boundary value problem:
\begin{equation}\label{3D-Free-prob-finite}
\left\{\begin{split}
&\mbox{\rm div}(\varrho^{\pm}(S^{\pm},{\bf q}^{\pm}){\bf q}^{\pm})=0,\\
&-\Delta(\psi^{\pm}{\bf e}_{\theta})=G^{\pm}(S^{\pm},\Lambda^{\pm},\partial_rS^{\pm},\partial_r\Lambda^{\pm},\nabla\varphi^{\pm},{\bf t}^{\pm}){\bf e}_{\theta},\\
&\varrho^{\pm}(S^{\pm},{\bf q}^{\pm}){\bf q}^{\pm}\cdot\nabla S^{\pm}=0,\\
&\varrho^{\pm}(S^{\pm},{\bf q}^{\pm}){\bf q}^{\pm}\cdot\nabla \Lambda^{\pm}=0,
\end{split}\right.\quad\mbox{in}\quad\Omega_{f,L}^{\pm}
\end{equation}
with the boundary conditions
\begin{equation}\label{3D-Free-bd-finite}%%psi컨디션 체크
\left\{
\begin{split}
\partial_r\varphi^{\pm}=-v_{\rm en}^{\pm},\quad \partial_x\psi^{\pm}=0,\quad S^{\pm}=S_{\rm en}^{\pm},\quad\Lambda^{\pm}=rw_{\rm en}^{\pm}\quad&\mbox{on}\,\,\Gamma_{0}^{\pm},\\
\partial_r\varphi^+=0,\quad \partial_x\psi^+=0\quad&\mbox{on}\,\,\Gamma_{{\rm w},L},\\
\nabla\varphi^{\pm}\cdot{\bf n}_f^{\pm}+(\nabla\times\psi^{\pm}{\bf e}_{\theta})\cdot{\bf n}_f^{\pm}={0}\quad&\mbox{on}\,\,\mathfrak{C}_{f,L},\\
S^+(\varrho^+)^{\gamma}(S^{+},{\bf q}^+)=S^-(\varrho^-)^{\gamma}(S^{-},{\bf q}^-) \quad&\mbox{on}\,\,\mathfrak{C}_{f,L},\\
\varphi^{\pm}=\varphi^{\pm}_0(L,\cdot),\quad \partial_x\psi^{\pm}=0\quad&\mbox{on}\,\,\Gamma_{f,L}^{\pm}
\end{split}
\right.
\end{equation}
for $\varrho^{\pm}$, $G^{\pm}$, and $\varphi_0^{\pm}$ given in \eqref{3D-varphi0}.
\end{problem}

%%%%%%%%%
 
A crucial difference from \cite{bae2019contact3D} is that the boundary data are no longer generated by a prescribed background state.
Instead, they arise from the pressure matching condition between two unknown Euler phases.
This feature introduces an additional nonlinear coupling that is absent in the one-phase setting.
%%%%%%%%%%

%%%%%%%%%%%%%%%%%%%%%%%%%%%%%%%%%%%%%%

\begin{proposition}\label{Theorem-HD}
For a fixed $\alpha\in(0,1)$,
there exist small constants $\varepsilon_1^{\star}>0$ depending only on ($\rho_0^{\pm}$, $u_0^{-}$, $p_0$, $\gamma$, $\alpha$) and  $\sigma_1^{\star}>0$ depending only on ($\rho_0^{\pm}$, $u_0^{\pm}$, $p_0$, $\gamma$, $\alpha$)  but independent of $L$ so that if 
\begin{equation*}
u_0^+\le\varepsilon_1^{\star}\mbox{ and }
\sigma\le \sigma_1^{\star},
\end{equation*}
then Problem \ref{H-free} has a unique solution $(f,S^{\pm},\Lambda^{\pm},\varphi^{\pm},\psi^{\pm})$ that satisfies 
\begin{equation}\label{3D-est-Thm-HD-finite}
\begin{split}
\|f-\frac{1}{2}\|_{2,\alpha,(0,L)}+\|S^{\pm}-S_0^{\pm}\|_{1,\alpha,\Omega^{\pm}_{f,L}}+\|\frac{\Lambda^{\pm}}{r}{\bf e}_{\theta}\|_{1,\alpha,\Omega^{\pm}_{f,L}}&\\
+\|\varphi^{\pm}-\varphi_0^{\pm}\|_{2,\alpha,\Omega^{\pm}_{f,L}}+\|\psi^{\pm}{\bf e}_{\theta}\|_{2,\alpha,\Omega^{\pm}_{f,L}}&
\le C\sigma
\end{split}
\end{equation}
for a positive constant $C$ depending only on the data %$\rho_0^{\pm}$, $u_0^{\pm}$, $p_0$, $\gamma$, and $\alpha$ 
but independent of $L$.
In particular, $\varepsilon_1^{\star}$ is independent of both the truncation length $L$ and the value of $u_0^+$.
\end{proposition}

%%%%%%%%%%%%%%%%%%%%%%%%%%%%%%%%%%%%%%

%\newpage
\subsection{Proof of Proposition \ref{Theorem-HD}}\label{S-proof-1}

The key new ingredient in the proof of Proposition 4.1 is the construction of a coupled iteration map for the free boundary and the two Euler phases.
Unlike the one-phase framework of \cite{bae2019contact3D}, the pressure continuity condition across the contact discontinuity must be enforced throughout the iteration procedure. This requires a simultaneous update of the free boundary and the flow states on both sides of the interface.

Fix $\alpha\in(0,1)$.
%%%%%%
For positive constants $\delta_1$  to be determined later, we define two iteration sets $\mathcal{I}(\delta_1)$  by 
\begingroup
\begin{align}\label{3D-Iteration-set}\allowdisplaybreaks %%%%%%Lambda set 수정필요\frac{1}{2}\mp\frac{1}{4}
\begin{aligned}
&\mathcal{I}(\delta_1):=\mathcal{I}_S(\delta_1)\times\mathcal{I}_{\Lambda}(\delta_1),\\
&\quad\mathcal{I}_{S}(\delta_1):=\left\{S^{-}\in C^{1,\alpha}(\overline{\Omega_{\frac{1}{2}+\frac{1}{4},L}^{-}})\left|\,\begin{aligned}
											&S^{-}\mbox{ is axisymmetric},\\
											%&S^{\pm}=S_{\rm en}^{\pm}\mbox{ on }\Gamma_{0}^{\pm},\\
											&\partial_x S^{-}=0\mbox{ on }\Gamma_{0,\epsilon}^{-}\cup\Gamma_{\frac{1}{2}+\frac{1}{4},L}^{-},\\
											&\|S^{-}-S_0^{-}\|_{1,\alpha,\Omega^{-}_{\frac{1}{2}+\frac{1}{4},L}}\le \delta_1\sigma
											\end{aligned}\right.\right\},\\
&\quad\mathcal{I}_{\Lambda}(\delta_1):=\left\{\frac{\Lambda^{-}}{r}{\bf e}_{\theta}\in C^{1,\alpha}(\overline{\Omega_{\frac{1}{2}+\frac{1}{4},L}^{-}})\left|\,\begin{aligned}
											&\Lambda^{-}\mbox{ is axisymmetric},\\
											%&\Lambda^{\pm}=rw_{\rm en}^{\pm}\mbox{ on }\Gamma_{0}^{\pm},\\
											&\partial_x \Lambda^{-}=0\mbox{ on }\Gamma_{0,\epsilon}^{-}\cup\Gamma_{\frac{1}{2}+\frac{1}{4},L}^{-},\\
											&\left\|\frac{\Lambda^{-}}{r}{\bf e}_{\theta}\right\|_{1,\alpha,\Omega^{-}_{\frac{1}{2}+\frac{1}{4},L}}\le \delta_1\sigma
											\end{aligned}\right.\right\}.					
\end{aligned}
\end{align}
\endgroup
%\begin{remark}
This is a free-boundary problem and the location of the contact discontinuity is unknown before the problem is solved. Therefore we need the fixed larger domains $\Omega_{\frac{1}{2}+\frac{1}{4},L}^{-}$ in \eqref{3D-Iteration-set}.
%\end{remark}

%%%%%%%%%%%%%%%%
For each of $(S_{\ast}^{-},\frac{\Lambda_{\ast}^{-}}{r}{\bf e}_{\theta})\in\mathcal{I}(\delta_1)$ and $(S_{\ast}^+,\frac{\Lambda^+}{r}{\bf e}_{\theta})=(S_0^+,{\bf 0})$, we solve the following free boundary problem for $(f,\varphi^{\pm},\psi^{\pm}{\bf e}_{\theta})$:
\begin{equation}\label{Free-1-prob}
\left\{\begin{split}
&\mbox{\rm div}(\varrho^{\pm}(S_{\ast}^{\pm},{\bf q}_{\ast}^{\pm}){\bf q}_{\ast}^{\pm})=0,\\
&-\Delta(\psi^{\pm}{\bf e}_{\theta})=G^{\pm}(S_{\ast}^{\pm},\Lambda_{\ast}^{\pm},\partial_rS_{\ast}^{\pm},\partial_r\Lambda_{\ast}^{\pm},\nabla\varphi^{\pm},{\bf t}_{\ast}^{\pm}){\bf e}_{\theta},\\
\end{split}\right.\quad\mbox{in}\quad\Omega_{f,L}^{\pm}
\end{equation}
with boundary conditions
\begin{equation}\label{Free-1-bd}
\left\{
\begin{split}
\partial_r\varphi^{\pm}=-v_{\rm en}^{\pm},\quad \partial_x\psi^{\pm}{\bf e}_{\theta}={\bf 0}\quad&\mbox{on}\,\,\Gamma_{0}^{\pm},\\
\partial_r\varphi^+=0,\quad \partial_x\psi^+{\bf e}_{\theta}={\bf 0}\quad&\mbox{on}\,\,\Gamma_{{\rm w},L},\\
\nabla\varphi^{\pm}\cdot{\bf n}_f^{\pm}+(\nabla\times\psi^{\pm}{\bf e}_{\theta})\cdot{\bf n}_f^{\pm}={0}\quad&\mbox{on}\,\,\mathfrak{C}_{f,L},\\
S_{\ast}^+(\varrho^+)^{\gamma}(S_{\ast}^{+},{\bf q}_{\ast}^+)=S_{\ast}^-(\varrho^-)^{\gamma}(S_{\ast}^{-},{\bf q}_{\ast}^-) \quad&\mbox{on}\,\,\mathfrak{C}_{f,L},\\
\varphi^{\pm}=\varphi^{\pm}_0(L,\cdot),\quad \partial_x\psi^{\pm}{\bf e}_{\theta}={\bf 0}\quad&\mbox{on}\,\,\Gamma_{f,L}^{\pm}.
\end{split}
\right.
\end{equation}
In the above, ${\bf q}_{\ast}^{\pm}$ and ${\bf t}_{\ast}^{\pm}$ are given by 
\begin{equation*}
{\bf q}_{\ast}^{\pm}:={\bf q}(\nabla\varphi^{\pm},\nabla\times(\psi^{\pm}{\bf e}_{\theta}),\frac{\Lambda_{\ast}^{\pm}}{r}{\bf e}_{\theta})
\mbox{ and }{\bf t}_{\ast}^{\pm}:=\nabla\times(\psi^{\pm}{\bf e}_{\theta})+\frac{\Lambda^{\pm}_{\ast}}{r}{\bf e}_{\theta}
\end{equation*}
 for simplicity of notation.

%%%%%%%%%%%%%%%%
\begin{lemma}\label{pro-41}
Under the same assumptions as in Theorem \ref{main=Thm-infty},
there exist small constants $\varepsilon_1^{\diamond}>0$ depending only on ($\rho_0^{\pm}$, $u_0^{-}$, $p_0$, $\gamma$, $\alpha$, $\delta_1$) and  $\sigma_1^{\diamond}>0$ depending only on ($\rho_0^{\pm}$, $u_0^{\pm}$, $p_0$, $\gamma$, $\alpha$, $\delta_1$)  but independent of $L$ so that if 
\begin{equation}\label{diamond}
u_0^+\le\varepsilon_1^{\diamond}\mbox{ and }
\sigma\le\sigma_1^{\diamond},
\end{equation}
then, for fixed $(S_{\ast}^{-},\frac{\Lambda_{\ast}^{-}}{r}{\bf e}_{\theta})\in\mathcal{I}(\delta_1)$ and $(S_{\ast}^+,\frac{\Lambda^+}{r}{\bf e}_{\theta})=(S_0^+,{\bf 0})$, the free boundary problem \eqref{Free-1-prob}-\eqref{Free-1-bd} has a unique axisymmetric solution $(f,\varphi^{\pm},\psi^{\pm}{\bf e}_{\theta})$ that satisfies 
\begin{equation}\label{3D-est-prop-free}
\begin{split}
\|f-\frac{1}{2}\|_{2,\alpha,(0,L)}&+\|\varphi^{+}-\varphi_0^{+}\|_{2,\alpha,\Omega_{f,L}^{+}}+\|\psi^{+}{\bf e}_{\theta}\|_{2,\alpha,\Omega_{f,L}^{+}}\\
&+\|\varphi^{-}-\varphi_0^{-}\|_{2,\alpha,\Omega_{f,L}^{-}}+\|\psi^{-}{\bf e}_{\theta}\|_{2,\alpha,\Omega_{f,L}^{-}}
\le C(1+\delta_1)\sigma
\end{split}
\end{equation}
for a positive constant $C$ depending only on the data but independent of $L$. In particular, $\varepsilon_1^{\diamond}$ is independent of both the truncation length $L$ and the value of $u_0^+$.
\end{lemma}
%%%%%%%%%%%%%%%%

\begin{proof}[Proof of Lemma \ref{pro-41}]
For a positive constant $\delta_2$ to be determined later, define an iteration set $\mathcal{I}(\delta_2)$ by 
\begin{equation*}
\mathcal{I}(\delta_2):=\left\{f\in C^{2,\alpha}([0,L])\left|\, \begin{split}
										&f(0)=\frac{1}{2}, f'(0)=f'(L)=0,\\
										&\|f-\frac{1}{2}\|_{2,\alpha,(0,L)}\le \delta_2\sigma\end{split}\right.\right\}.
\end{equation*}
Fix $f_{\ast}\in\mathcal{I}(\delta_2)$. If it holds that $\sigma\le \frac{1}{ 4 \delta_2}$, then $|f_{\ast}(x_1)-\frac{1}{2}|\le\frac{1}{4}$ for $x\in[0,L]$.
For such a function $f_{\ast}$, we solve the following fixed boundary problem for $(\varphi^{\pm},\psi^{\pm}{\bf e}_{\theta})$ in $\Omega_{f_{\ast},L}^{\pm}$:
\begin{equation}\label{Fix-prob}
\left\{\begin{split}
&\mbox{\rm div}(\varrho^{\pm}(S_{\ast}^{\pm},{\bf q}_{\ast}^{\pm}){\bf q}_{\ast}^{\pm})=0,\\
&-\Delta(\psi^{\pm}{\bf e}_{\theta})=G^{\pm}(S_{\ast}^{\pm},\Lambda_{\ast}^{\pm},\partial_rS_{\ast}^{\pm},\partial_r\Lambda_{\ast}^{\pm},\nabla\varphi^{\pm},{\bf t}_{\ast}^{\pm}){\bf e}_{\theta},\\
%&\varrho^{\pm}(S^{\pm},{\bf q}^{\pm}){\bf q}^{\pm}\cdot\nabla S^{\pm}=0,\\
%&\varrho^{\pm}(S^{\pm},{\bf q}^{\pm}){\bf q}^{\pm}\cdot\nabla \Lambda^{\pm}=0,
\end{split}\right.\,\,\mbox{in}\,\,\Omega_{f_{\ast},L}^{\pm}
\end{equation}
with the boundary conditions
\begin{equation}\label{Fix-bd}
\left\{
\begin{split}
\varphi^{+}=-\int_{\frac{1}{2}}^rv_{\rm en}^{+}(t)dt,\quad \partial_x\psi^{+}{\bf e}_{\theta}={\bf 0}%,\quad S^{+}=S_{\rm en}^{+},\quad\Lambda^{+}=rw_{\rm en}^{+}
\quad&\mbox{on}\,\,\Gamma_{0}^{+},\\
\partial_r\varphi^+=0,\quad \partial_x\psi^+{\bf e}_{\theta}={\bf 0}\quad&\mbox{on}\,\,\Gamma_{{\rm w},L},\\
\nabla\varphi^{+}\cdot{\bf n}_{f_{\ast}}^{+}=0,\quad (\nabla\times\psi^{+}{\bf e}_{\theta})\cdot{\bf n}_{f_{\ast}}^{+}={0}\quad&\mbox{on}\,\,\partial\Omega_{f_{\ast},L}^+\cap\mathfrak{C}_{f_{\ast},L},\\
\varphi^{+}=\varphi^{+}_0(L,\cdot),\quad \partial_x\psi^{+}{\bf e}_{\theta}={\bf 0}\quad&\mbox{on}\,\,\Gamma_{f_{\ast},L}^{+}
\end{split}
\right.
\end{equation}
and 
\begin{equation}\label{Fix-bd-lower}
\left\{
\begin{split}
\varphi^{-}=-\int_{\frac{1}{2}}^r v_{\rm en}^{-}(t)dt,\quad \partial_x\psi^{-}{\bf e}_{\theta}={\bf 0}%,\quad S^{-}=S_{\rm en}^{-},\quad\Lambda^{-}=rw_{\rm en}^{-}
\quad&\mbox{on}\,\,\Gamma_{0}^{-},\\
S_{\ast}^-(\varrho^-)^{\gamma}(S_{\ast}^{-},{\bf q}_{\ast}^-)=S_{\ast}^+(\varrho^+)^{\gamma}(S_{\ast}^{+},{\bf q}_{\ast}^+) \quad&\mbox{on}\,\,\partial\Omega_{f_{\ast},L}^-\cap\mathfrak{C}_{f_{\ast},L},\\
\varphi^{-}=\varphi^{-}_0(L,\cdot),\quad \partial_x\psi^{-}{\bf e}_{\theta}={\bf 0}\quad&\mbox{on}\,\,\Gamma_{f_{\ast},L}^{-}.
\end{split}
\right.
\end{equation}
%%%%%%%%%%%%%%%%%%%%%%%%%%%%%%%%%
\begin{remark}
%%%%%%%%%
We need to decompose the boundary condition 
\begin{equation}\label{bd-pp}
S_{\ast}^-(\varrho^-)^{\gamma}(S_{\ast}^{-},{\bf q}_{\ast}^-)=S_{\ast}^+(\varrho^+)^{\gamma}(S_{\ast}^{+},{\bf q}_{\ast}^+) \quad\mbox{on}\,\,\partial\Omega_{f_{\ast},L}^-\cap\mathfrak{C}_{f_{\ast},L}
\end{equation}
into boundary conditions for $\varphi^-$ and $\psi^-{\bf e}_{\theta}$.
To simplify notation, let us set 
\begin{equation*}
\mathfrak{p}(\zeta,{\bf q}):=\zeta\varrho^{\gamma}(\zeta,{\bf q})
\end{equation*}
for $\zeta\in\mathbb{R}$ and ${\bf q}\in\mathbb{R}^3$.
Then \eqref{bd-pp} is equivalent to 
\begin{equation*}
\mathfrak{p}(S_{\ast}^-,{\bf q}_{\ast}^-)=\mathfrak{p}(S_{\ast}^+,{\bf q}_{\ast}^+).
\end{equation*}
By the definition of the Bernoulli invariant and \eqref{bd-pp}, we have
\begin{equation}\label{bd-ber}
%\begin{split}
|{\bf q}_{\ast}^-|^2
=2\left(B_0^--\frac{\gamma}{\gamma-1}(\mathfrak{p}(S_{\ast}^+,{\bf q}_{\ast}^+))^{1-\frac{1}{\gamma}}(S_{\ast}^-)^{\frac{1}{\gamma}}\right)
=: \mathfrak{Y}
%\end{split}
\end{equation}
If it  holds that  ${\bf q}_{\ast}^-\cdot{\bf n}^-_{f_\ast}=0$ on $\partial\Omega_{f_{\ast},L}^-\cap\mathfrak{C}_{f_{\ast},L}$, % for the outward unit normal vector field ${\bf n}^+_{f_{\ast}}:=\frac{(f_{\ast}',-1)}{\sqrt{1+|f_{\ast}'|^2}}$ on $\partial\Omega_{f_{\ast}}^+\cap\mathfrak{C}_{f_{\ast}}$, 
then \eqref{bd-ber} can be decomposed into conditions for $\varphi^-$ and $\psi^-{\bf e}_{\theta}$ as follows:
\begin{equation}\label{dec-bd}
\begin{split}
&\varphi^-=\varphi_0^-,\\
&(\nabla\times \psi^-{\bf e}_{\theta})\cdot{\bm\tau}_{f_{\ast}}=\sqrt{\mathfrak{Y}-\left(\frac{\Lambda_{\ast}}{f_{\ast}}\right)^2}-\nabla\varphi_0^-\cdot{\bm\tau}_{f_{\ast}}
%&\Lambda=rw_{\rm en}^+(0)
\mbox{ on }\partial\Omega_{f_{\ast},L}^-\cap\mathfrak{C}_{f_{\ast},L}
\end{split}
\end{equation}
for a tangential vector field ${\bm\tau}_{f_{\ast}}$ defined by 
\begin{equation*}
{\bm\tau}_{f_{\ast}}=\frac{{\bf e}_x+f'_{\ast}(x){\bf e}_r}{\sqrt{1+|f_{\ast}'|^2}}.
\end{equation*}
We will find ${\bf q}_{\ast}^-$ satisfying \eqref{dec-bd} and ${\bf q}_{\ast}^-\cdot{\bf n}_{f_\ast}^-=0$ on $\partial\Omega_{f_{\ast},L}^-\cap\mathfrak{C}_{f_{\ast},L}$.

In the one-phase framework of \cite{bae2019contact3D}, the pressure continuity condition is partially absorbed by the prescribed background state.
In the present two-phase setting, the pressure continuity condition must be enforced between two unknown Euler phases.
The decomposition \eqref{dec-bd} converts the pressure continuity condition, together with the slip condition on the interface, into boundary conditions for the Helmholtz variables, which makes it possible to incorporate the Rankine--Hugoniot condition directly into the coupled iteration scheme.

\end{remark}
%%%%%%%%%%%%%%%%%%%%%%%%%%%%%%%%%

\begin{lemma}\label{Lemma-fixed}
Under the same assumptions as in Theorem \ref{main=Thm-infty},
there exists a small constant $\sigma_2>0$ depending only on the data and $(\delta_1,\delta_2)$ but independent of $L$ so that if 
\begin{equation}\label{sigma2}
\sigma\le\sigma_2,
\end{equation}
then, for each $f_{\ast}\in\mathcal{I}(\delta_2)$, the fixed boundary problem \eqref{Fix-prob}-\eqref{Fix-bd-lower} has a unique axisymmetric solution $(\varphi^{\pm},\psi^{\pm}{\bf e}_{\theta})\in \left[C^{2,\alpha}(\overline{\Omega_{f_\ast,L}^{\pm}})\right]^4$ that satisfies 
\begin{equation}\label{3D-est-prop}
\begin{split}
&\|\varphi^{+}-\varphi_0^{+}\|_{2,\alpha,\Omega_{f_{\ast},L}^{+}}+\|\psi^{+}{\bf e}_{\theta}\|_{2,\alpha,\Omega_{f_{\ast},L}^{+}}\le C(1+u_0^+\delta_2)\sigma,\\
&\|\varphi^{-}-\varphi_0^{-}\|_{2,\alpha,\Omega_{f_{\ast},L}^{-}}+\|\psi^{-}{\bf e}_{\theta}\|_{2,\alpha,\Omega_{f_{\ast},L}^{-}}\le C(1+\delta_1+u_0^+\delta_2)\sigma%%
\end{split}
\end{equation}
for a positive constant $C$ depending only on the data but independent of $L$. 
\end{lemma}

The proof of Lemma \ref{Lemma-fixed} is given in the next section. 
%%
%%%
Assume that Lemma \ref{Lemma-fixed} is true and 
define an iteration mapping $\mathcal{J}_{\rm cd}:\mathcal{I}(\delta_2)\to C^{2,\alpha}([0,L])$ by 
\begin{equation}\label{J-cd}
\mathcal{J}_{\rm cd}:f_{\ast}\mapsto f,
\end{equation}
where $f:[0,L]\to\mathbb{R}$ is a function satisfying
\begin{equation}\label{S-curve}
\begin{split}
\int_{f(x)}^{f_{\ast}(x)}t\rho_0^-u_{0}^- dt
=&\int^{\frac{1}{2}}_0t\varrho^{-}(S_{\ast}^{-},{\bf q}^{-})({\bf q}^-\cdot{\bf e}_x)(0,t)dt
-\int^{f_{\ast}(x)}_0t\varrho^{-}(S_{\ast}^{-},{\bf q}^-)({\bf q}^-\cdot{\bf e}_x)(x,t)dt
\end{split}
\end{equation}
for ${\bf q}^-={\bf q}(\nabla\varphi^{-},\nabla\times(\psi^{-}{\bf e}_{\theta}),\frac{\Lambda_{\ast}^{-}}{r}{\bf e}_{\theta})$  obtained in Lemma \ref{Lemma-fixed} associated with $f_{\ast}$.
Then, obviously, $f(0)=\frac{1}{2}$.
Since $\rho_0^-u_0^->0$, 
\eqref{S-curve} is equivalent to 
\begin{equation*}
f(x)=\sqrt{f_{\ast}^2(x)-\frac{2}{\rho_0^-u_0^-}\mbox{(RHS of \eqref{S-curve})}}.
\end{equation*}
Since $\|f_{\ast}-\frac{1}{2}\|_{2,\alpha,(0,L)}\le \delta_2\sigma$, 
there exists a small $\sigma^{\natural}\in(0,\sigma_2]$ depending only on the data and $(\delta_1,\delta_2)$ but independent of $L$ so that if $$\sigma\le\sigma^{\natural},$$ then $f$ is well-defined. 
Also,  we have 
\begin{equation*}
f'(0)=f'(L)=0
\end{equation*}
and
\begin{equation}\label{ff-est}
\|f-\frac{1}{2}\|_{2,\alpha,(0,L)}\le C_{\star}(1+\delta_1+u_0^+\delta_2)\sigma
\end{equation}
for a positive constant $C_{\star}$ depending only on the data but independent of $L$.
If $f= f_{\ast}$, then we differentiate \eqref{S-curve} with respect to $x$ and use the continuity equation to get 
\begin{equation*}
f'(x)=\frac{{\bf q}^-\cdot{\bf e}_r}{{\bf q}^-\cdot{\bf e}_x}(x,f(x)).
\end{equation*}
This means that ${\bf q}^-\cdot{\bf n}_f^-=0$ on $\partial\Omega_{f,L}^{-}\cap\mathfrak{C}_{f,L}$.
%%%%%%%%%%%%%%5
Thus, if we choose $\delta_2$, $\varepsilon_1^{\diamond}$, and $\sigma_1^{\diamond}$  satisfying
\begin{equation}\label{def-delta2}
\begin{split}
&\delta_2:=2C_{\star}(1+\delta_1),\,\, \varepsilon_1^{\diamond}\le \frac{1}{2C_{\star}},\mbox{ and }\sigma_1^{\diamond}\le \min\left\{\sigma^{\natural},\sigma_2, \frac{1}{4\delta_2}\right\}=:\sigma^{\sharp},
\end{split}
\end{equation}
then $f\in \mathcal{I}(\delta_2)$.
%%%%%%%%%%%%%%%%%%

The iteration set $\mathcal{I}(\delta_2)$ is a convex and compact subset of $C^{2,\alpha/2}([0,L])$.
%To prove the continuity of the mapping $\mathcal{J}_{\rm cd}$, 
Assume that $\{f_k^{\ast}\}_{k=1}^{\infty}\subset\mathcal{I}(\delta_2)$ converges in $C^{2,\alpha/2}([0,L])$ to $f_{\infty}^{\ast}\in\mathcal{I}(\delta_2)$. 
For each $k\in\mathbb{N}\cup\{\infty\}$, let us set 
\begin{equation*}
f_k:=\mathcal{J}_{\rm cd}(f_k^{\ast})
\end{equation*}
and let $\mathcal{U}_k^{\pm}:=(\varphi_k^{\pm},\psi_k^{\pm}{\bf e}_{\theta})$ be the unique solution to the fixed boundary problem \eqref{Fix-prob}-\eqref{Fix-bd-lower} associated with $f_{\ast}=f_k^{\ast}$.
For the transformations $T_k^{\pm}:\overline{\Omega^{\pm}_{f_k^{\ast},L}}\to\overline{\Omega^{\pm}_{f_\infty^{\ast},L}}$  defined by 
\begin{equation*}
\begin{split}
&T_k^+:(x,r,\theta)\mapsto\left(x,\frac{1-f_\infty^{\ast}(x)}{1-f_k^{\ast}(x)}(r-1)+1,\theta\right),\\
&T_k^-:(x,r,\theta)\mapsto \left(x,\frac{f_{\infty}^{\ast}(x)}{f_k^{\ast}(x)}r,\theta\right),
\end{split}
\end{equation*}
$\left\{\mathcal{U}_k^{\pm}\circ (T_k^{\pm})^{-1}\right\}_{k=1}^{\infty}$ is sequentially compact in  $\left[C^{2,\alpha/2}(\overline{\Omega_{f_\infty^\ast,L}^{\pm}})\right]^4$ and the limit of each convergent subsequence of $\left\{\mathcal{U}_k^{\pm}\circ (T_k^{\pm})^{-1}\right\}_{k=1}^{\infty}$ in $\left[C^{2,\alpha/2}(\overline{\Omega_{f_\infty^\ast,L}^{\pm}})\right]^4$ solves the fixed boundary problem \eqref{Fix-prob}-\eqref{Fix-bd-lower} associated with $f_{\ast}=f_\infty^{\ast}$. 
The uniqueness of a solution to the problem \eqref{Fix-prob}-\eqref{Fix-bd-lower} implies that $\left\{\mathcal{U}_k^{\pm}\circ (T_k^{\pm})^{-1}\right\}_{k=1}^{\infty}$ is convergent in $\left[C^{2,\alpha/2}(\overline{\Omega_{f_\infty^\ast,L}^{\pm}})\right]^4$, which provides that $f_k$ converges to $f_{\infty}$ in $C^{2,\alpha/2}([0,L])$. Thus the mapping $\mathcal{J}_{\rm cd}$ is continuous in $C^{2,\alpha/2}([0,L])$.
%%%%
Then, according to the Schauder fixed point theorem, $\mathcal{J}_{\rm cd}$ has a fixed point $f_{\natural}\in \mathcal{I}(\delta_2)$.
For such a fixed point $f_{\natural}:[0,L]\to(\frac{1}{4},\frac{3}{4})$, we have a unique solution $(\varphi^{\pm}_{\natural},\psi^{\pm}_{\natural}{\bf e}_{\theta})$ to the fixed boundary problem \eqref{Fix-prob}-\eqref{Fix-bd}  associated with $f_{\ast}=f_{\natural}$ by Lemma \ref{Lemma-fixed}.
Then $(f_{\natural},\varphi^{\pm}_{\natural},\psi^{\pm}_{\natural}{\bf e}_{\theta})$ is the solution to the free boundary problem \eqref{Free-1-prob}-\eqref{Free-1-bd}.
From \eqref{3D-est-prop}, \eqref{ff-est}, and \eqref{def-delta2}, we can obtain the estimate \eqref{3D-est-prop-free}.
%%%
 
 To prove uniqueness, let $(f_k,\varphi^{\pm}_k,\psi^{\pm}_k{\bf e}_{\theta})$ $(k=1,2$) be two solutions to the free boundary problem \eqref{Free-1-prob}-\eqref{Free-1-bd} and suppose that they satisfy the estimate  \eqref{3D-est-prop-free}.
For transformations $\mathfrak{T}^{\pm}_d:\overline{\Omega_{f_1,L}^{\pm}}\to\overline{\Omega_{f_2,L}^{\pm}}$ defined by 
\begin{equation}\label{domain-trans}
\begin{split}
&\mathfrak{T}_d^{+}:(x,r,\theta)\mapsto\left(x,\frac{1-f_2(x)}{1-f_1(x)}(r-1)+1,\theta\right),\\
&\mathfrak{T}_d^{-}:(x,r,\theta)\mapsto\left(x,\frac{f_2(x)}{f_1(x)}r,\theta\right),
\end{split}
\end{equation}
we subtract the elliptic systems and boundary conditions of $(\varphi^{\pm}_1,\psi^{\pm}_1)$ and $(\varphi_2^{\pm}\circ(\mathfrak{T}_d^{\pm})^{-1},\psi_2^{\pm}\circ(\mathfrak{T}_d^{\pm})^{-1})$ in the fixed domains $\Omega_{f_1,L}^{\pm}$ and obtain the estimates of 
\begin{equation*}
d\varphi^{\pm}:=\varphi^{\pm}_1-\left(\varphi_2^{\pm}\circ(\mathfrak{T}_d^{\pm})^{-1}\right)\mbox{ and }d\psi^{\pm}:=\psi^{\pm}_1-\left(\psi_2^{\pm}\circ(\mathfrak{T}_d^{\pm})^{-1}\right).
\end{equation*}
Similarly to \eqref{3D-est-prop}, one can see that there exists a small $\sigma_d\in(0,\sigma^{\sharp}]$ depending only on the data and $\delta^{\pm}_1$  but independent of $L$ so that if $\sigma\le\sigma_d$ then we have 
\begin{equation}\label{d-plus-est}
\|d\varphi^+\|_{2,\alpha,\Omega_{f_1,L}^+}+\|d\psi^+\|_{2,\alpha,\Omega_{f_1,L}^+}
\le C_1\left(\sigma+u_0^+\right)\|f_1-f_2\|_{2,\alpha,(0,L)}
\end{equation}
and
\begin{equation}\label{d-minus-est}
\begin{split}
\|d\varphi^-\|_{2,\alpha,\Omega_{f_1,L}^-}+\|d\psi^-\|_{2,\alpha,\Omega_{f_1,L}^-}
\le &\,C_2\left((\delta_1+1)\sigma+u_0^+\right)\|f_1-f_2\|_{2,\alpha,(0,L)}\\
&+C_3\left(\|d\varphi^+\|_{2,\alpha,\Omega_{f_1,L}^+}+\|d\psi^+\|_{2,\alpha,\Omega_{f_1,L}^+}\right)
\end{split}
\end{equation}
for constants $C_i$ ($i=1,2,3$) depending only on the data but independent of $L$.
By using \eqref{S-curve}, we compute the estimate of $f_1-f_2$. 
There exists a small $\sigma_e\in(0,\sigma_d]$ depending only on the data and $\delta^{\pm}_1$ but independent of $L$ so that if $\sigma\le\sigma_e$, then 
\begin{equation}\label{f-d-est}
\begin{split}
\|f_1-f_2\|_{2,\alpha,(0,L)}
\le &\, C\left(\|d\varphi^-\|_{2,\alpha,\Omega_{f_1,L}^-}+\|d\psi^-\|_{2,\alpha,\Omega_{f_1,L}^-}\right).
\end{split}
\end{equation}
for a constant $C$ depending only on the data but independent of $L$.
Combining \eqref{d-plus-est}-\eqref{f-d-est} gives 
\begin{equation}\label{f12-unique}
\|f_1-f_2\|_{2,\alpha,(0,L)}\le C_{\dagger}\left((\delta_1+1)\sigma+u_0^+\right)\|f_1-f_2\|_{2,\alpha,(0,L)}
\end{equation}
for a constant $C_{\dagger}$ depending only on the data but independent of $L$.
Since all estimate constants depend on $u_0^+$ only through nonnegative powers, after fixing the remaining parameters, we can choose
$\varepsilon_1^{\diamond}>0$, independent of $u_0^+$, sufficiently small so that if $u_0^+\le\varepsilon_1^{\diamond}$, then 
\begin{equation*}
u_0^+\le \min\left\{\frac{1}{2C_{\star}},\frac{1}{2C_{\dagger}}\right\}.
\end{equation*}
We choose $\sigma_1^{\diamond}$ as 
\begin{equation*}
%\varepsilon_1^{\diamond}\le \min\left\{\frac{1}{2C_{\star}},\frac{1}{2C_{\dagger}}\right\}
%\mbox{ and }
\sigma_1^{\diamond}=\min\left\{\sigma_e,\frac{1}{2C_{\dagger}(\delta_1+1)}\right\}.
\end{equation*}
Then \eqref{f12-unique} implies that $f_1=f_2$. 
Substituting this into \eqref{d-plus-est} and \eqref{d-minus-est}, we obtain \(d\varphi^\pm=d\psi^\pm=0\). Hence the solution is unique.
The proof of Lemma \ref{pro-41} is completed.
\end{proof}

%%%%%%%%%%%%%%%%%%%%%%%%%%%%%%%%%
%%%%%%%%%%%%%%%%%%%%%%%%%%%%%%%%%
%%%%%%%%%%%%%%%%%%%%%%%%%%%%%%%%%
%%%%%%%%%%%%%%%%%%%%%%%%%%%%%%%%%%%%%%%%%%%%%%
%%%%%%%%%%%%%%%%%%%%%%%%%%%%%%%%%%%%%%%%%%%%%%

For the solution $(f_{\natural},\varphi^{\pm}_{\natural},\psi^{\pm}_{\natural}{\bf e}_{\theta})$ obtained in Lemma \ref{pro-41}, let us set 
$${\bf q}_{\natural}^{\pm}:={\bf q}(\nabla\varphi_{\natural}^{\pm},\nabla\times(\psi^{\pm}_{\natural}{\bf e}_{\theta}),\frac{\Lambda_{\ast}^{\pm}}{r}{\bf e}_{\theta})$$
and  solve the following initial value problem for $(S^{\pm},\Lambda^{\pm})$ in $\Omega_{f_{\natural},L}^{\pm}$:
\begin{equation}\label{lem-ini}
\left\{\begin{split}
\varrho^{\pm}(S_{\ast}^{\pm},{\bf q}_{\natural}^{\pm}){\bf q}_{\natural}^{\pm}\cdot\nabla S^{\pm}=0\quad&\mbox{in}\,\,\Omega_{f_{\natural},L}^{\pm},\\
S^{\pm}=S_{\rm en}^{\pm}\quad&\mbox{on}\,\,\Gamma_{0}^{\pm},\\
\varrho^{\pm}(S_{\ast}^{\pm},{\bf q}_{\natural}^{\pm}){\bf q}_{\natural}^{\pm}\cdot\nabla \Lambda^{\pm}=0\quad&\mbox{in}\,\,\Omega_{f_{\natural},L}^{\pm},\\
\Lambda^{\pm}=rw_{\rm en}^{\pm}\quad&\mbox{on}\,\,\Gamma_{0}^{\pm}.
\end{split}\right.%\quad\mbox{in}\quad\Omega_{f,L}^{\pm}
\end{equation}
%%%%%%%
It follows from \eqref{Fix-prob} that 
\begin{equation}\label{trans-1}
\left\{\begin{split}
\mbox{div}\left(\varrho^{\pm}(S_{\ast}^{\pm},{\bf q}_{\natural}^{\pm}){\bf q}_{\natural}^{\pm}\right)=0&\mbox{ in }\Omega_{f_{\natural},L}^{\pm},\\
{\bf q}^{\pm}_{\natural}\cdot{\bf e}_r=0&\mbox{ on }\Gamma_{{\rm w},L}^{\pm},\\
{\bf q}^{\pm}_{\natural}\cdot{\bf n}_{f_{\natural}}^{\pm}=0&\mbox{ on }\partial\Omega_{f_{\natural},L}^{\pm}\cap\mathfrak{C}_{f_{\natural},L}.
\end{split}\right.
\end{equation}
Also, one can choose a small $\sigma_{\flat}\in(0,\sigma_1^{\diamond}]$ depending only on the data and $\delta_1$  but independent of $L$ so that if $\sigma<\sigma_{\flat}$, then 
there exists a constant $\mathfrak{x}>0$ depending only on the data  such that 
\begin{equation}\label{trans-2}
\rho^{\pm}(S_{\ast}^{\pm},{\bf q}_{\natural}^{\pm}){\bf q}_{\natural}^{\pm}\cdot{\bf e}_x>\mathfrak{x}>0\mbox{ in }\Omega_{f_{\natural},L}^{\pm}.
\end{equation}
Once the coupled free-boundary problem \eqref{Free-1-prob}-\eqref{Free-1-bd} has been solved and we obtain \eqref{trans-1}-\eqref{trans-2}, the problem \eqref{lem-ini} can be treated by the characteristic method developed in \cite{bae2019contact3D} which yields the following lemma.
%Then we can obtain the following lemma as in \cite[Lemma 4.4]{bae2019contact3D}. 
%We skip the proof of it.
 %For completeness, we only state the main estimates needed for the present two-phase setting.
%%%%%%%%%%%%%%%%%%%%%%%%%%%%%%%%%%%%%%%%%%%%%%%%%%%%%%%
%%%%%%%%%%%%%%%%%%%%%%%%%
%%%%%%%%%%%%%%%%%%%%%%
\begin{lemma}\label{Lem-trans} 
There exist small constants $\varepsilon_3\in(0,\varepsilon_1^{\diamond}]$  depending only on ($\rho_0^{\pm}$, $u_0^{-}$, $p_0$, $\gamma$, $\alpha$)  and  $\sigma_3\in(0,\sigma_{\flat}]$ depending only on ($\rho_0^{\pm}$, $u_0^{\pm}$, $p_0$, $\gamma$, $\alpha$)    but independent of $L$ so that if 
\begin{equation*}
u_0^+\le \varepsilon_3\mbox{ and }\sigma\le\sigma_3,
\end{equation*}
then, in each region of $\Omega_{f_{\natural},L}^{\pm}$,
the initial value problem \eqref{lem-ini} has a unique axisymmetric solution $(S^{\pm},\Lambda^{\pm})$
%\in C^{1,\alpha}(\overline{\Omega_{f_{\sharp},L}^{\pm}})\times C^{1,\alpha}(\overline{\Omega_{f_{\sharp},L}^{\pm}})$ 
that satisfies 
\begin{equation}\label{S-est-natural}
\|S^{\pm}-S_0^{\pm}\|_{1,\alpha,\Omega_{f_{\natural},L}^{\pm}}+\|\frac{\Lambda^{\pm}}{r}{\bf e}_{\theta}\|_{1,\alpha,\Omega_{f_{\natural},L}^{\pm}}\le C_{\natural}^{\pm}\sigma
\end{equation}
for constants $C_{\natural}^{\pm}>0$ depending only on the data but independent of $L$.
In particular, $\varepsilon_3$ is independent of both the truncation length $L$ and the value of $u_0^+$.
Moreover, $S^{\pm}$ and $\Lambda^{\pm}$ have the forms of 
\begin{equation}\label{S-form}
\begin{split}
&S^{\pm}(x,r):=S_{\rm en}^{\pm}\circ\mathcal{T}^{\pm}(x,r),\\
&\Lambda^{\pm}(x,r):=r{w_{\rm en}^{\pm}}\circ\mathcal{T}^{\pm}(x,r)
\end{split}
\end{equation}
for functions $\mathcal{T}^{\pm}:\overline{\Omega^{\pm}_{f_{\natural},L}}\to[0,1]$ satisfying
\begin{equation*}
\begin{split}
&\int_{\mathcal{T}^{+}(x,r)}^{1}t\varrho^{+}(S_{\ast}^{+},{\bf q}_{\natural}^+)({\bf q}_{\natural}^+\cdot{\bf e}_x)(0,t)dt
=\int_{r}^{1}t\varrho^{+}(S_{\ast}^{+},{\bf q}_{\natural}^+)({\bf q}_{\natural}^+\cdot{\bf e}_x)(x,t)dt,\\
&\int^{\mathcal{T}^{-}(x,r)}_0t\varrho^{-}(S_{\ast}^{-},{\bf q}_{\natural}^-)({\bf q}_{\natural}^-\cdot{\bf e}_x)(0,t)dt
=\int^{r}_{0}t\varrho^{-}(S_{\ast}^{-},{\bf q}_{\natural}^-)({\bf q}_{\natural}^-\cdot{\bf e}_x)(x,t)dt.
\end{split}
\end{equation*}
\end{lemma}

\begin{remark}
The condition \eqref{po-con} implies that $S^+=S_0^+$ and $\Lambda^+=0$. Moreover, the equation $\psi^+{\bf e}_{\theta}$ becomes homogeneous, and the boundary conditions imposed on $\psi^+{\bf e}_{\theta}$, together with the uniqueness of the corresponding elliptic problem, yield $\psi^+{\bf e}_{\theta}={\bf 0}$. 
Hence ${\bf u}^+=\nabla\varphi^+$, so the plus phase is a potential flow.
%Although $S^+$ and $\Lambda^+$ remain fixed, the potential phase is still unknown through $\varphi^+$, and the pressure matching condition couples $\varphi^+$, $(S^-,\Lambda^-,\varphi^-,\psi^-)$, and $f$.
\end{remark}
%%%%%%%%%%%%%%%%%%%%%%%
%%%%%%%%%%%%%%%%%%%%%%%%%%%%%%%%%%%%%%%%%%%%%%%%%%%%%%

For the solution $(S^{-},\Lambda^{-})$ obtained in Lemma \ref{Lem-trans},
let us define extensions of $S^{-}$ and $\Lambda^{-}$ as follows:
\begin{equation*}
\begin{split}
&S_{\rm ext}^-(x,r):=\left\{\begin{split}
					\tilde{S}_{\rm e}^-\circ T_{\rm e}^-(x,r)\mbox{ for }r>f_{\natural}(x),\\
					S^-(x,r)\mbox{ for }r<f_{\natural}(x),\\
					\end{split}\right.\quad
\Lambda_{\rm ext}^-(x,r):=\left\{\begin{split}
					\tilde{\Lambda}_{\rm e}^-\circ T_{\rm e}^-(x,r)\mbox{ for }r>f_{\natural}(x),\\
					\Lambda^-(x,r)\mbox{ for }r<f_{\natural}(x),\\
					\end{split}\right.\\
\end{split}
\end{equation*}
for $T_{\rm e}^{-}$, $\tilde{S}_{\rm e}^{-}$, and $\tilde{\Lambda}_{\rm e}^{-}$ defined by 
\begin{equation*}
\begin{split}
&T_{\rm e}^-:(x,r)\mapsto\left(x,-\frac{r}{2f_{\natural}(x)}+\frac{1}{2}\right),\\
&\tilde{S}_{\rm e}^{-}(y,s):=\sum_{i=1}^3c_i(S^{-}\circ (T_{\rm e}^{-})^{-1})\left(y,-\frac{s}{i}\right),\, s<0,\\
&\tilde{\Lambda}_{\rm e}^{-}(y,s):=\sum_{i=1}^3c_i(\Lambda^{-}\circ (T_{\rm e}^{-})^{-1})\left(y,-\frac{s}{i}\right),\, s<0,\\
\end{split}
\end{equation*}
with $c_i$ ($i=1,2,3$) satisfying $\sum_{i=1}^3c_i(-\frac{1}{i})^m=1$, $m=0,1,2$.
%%%%%%
Define iteration mappings $\mathcal{J}_t^{-}:\mathcal{I}(\delta_1)\to \left[C^{1,\alpha}(\overline{\Omega_{\frac{1}{2}+\frac{1}{4},L}^{-}})\right]^4$ by 
\begin{equation*}
\mathcal{J}_t^{-}:(S_{\ast}^{-},\frac{\Lambda_{\ast}^{-}}{r}{\bf e}_{\theta})\mapsto (S_{\rm ext}^{-},\frac{\Lambda_{\rm ext}^{-}}{r}{\bf e}_{\theta}).
\end{equation*}
%%%%%%%%
%%
We choose $\delta_1$  as 
\begin{equation*}
\delta_1:=C_{\natural}^{-}
\end{equation*}
for $C_{\natural}^{-}$ in \eqref{S-est-natural}
so that $(S_{\rm ext}^{-},\frac{\Lambda_{\rm ext}^{-}}{r}{\bf e}_{\theta})\in \mathcal{I}(\delta_1)$.
%%%%
%%
%%%%%%
Similarly to the mapping $\mathcal{J}_{\rm cd}$ in \eqref{J-cd}, we can use the Schauder fixed point theorem to prove that there exist  fixed points $(S_{\sharp}^{-},\frac{\Lambda_{\sharp}^{-}}{r}{\bf e}_{\theta})$ of $\mathcal{J}_t^{-}$.
For the fixed points $(S_{\sharp}^{-},\frac{\Lambda^{-}_{\sharp}}{r}{\bf e}_{\theta})$, let $(f_{\star},\varphi_{\star}^{\pm},\psi_{\star}^{\pm}{\bf e}_{\theta})$ be the solution to the free boundary problems \eqref{Free-1-prob}-\eqref{Free-1-bd} associated with $(S_{\ast}^{-},\frac{\Lambda_{\ast}^{-}}{r}{\bf e}_{\theta})=(S_{\sharp}^{-},\frac{\Lambda^{-}_{\sharp}}{r}{\bf e}_{\theta})$.
Then, for $(S_{\sharp}^{+},\frac{\Lambda^{+}_{\sharp}}{r}{\bf e}_{\theta})=(S_0^+,{\bf 0})$,  $(f_{\star}, S_{\sharp}^{\pm},\frac{\Lambda^{\pm}_{\sharp}}{r}{\bf e}_{\theta},\varphi^{\pm}_{\star},\psi_{\star}^{\pm}{\bf e}_{\theta})$ is the solution to Problem \ref{H-free} and satisfies the estimate \eqref{3D-est-Thm-HD-finite}.

To prove the uniqueness, let $(f_k, S_{k}^{\pm},\frac{\Lambda^{\pm}_{k}}{r}{\bf e}_{\theta},\varphi^{\pm}_{k},\psi_{k}^{\pm}{\bf e}_{\theta})$ be two solutions to Problem \ref{H-free} and suppose that they satisfy the estimate \eqref{3D-est-Thm-HD-finite}.
Similarly to the proof of Lemma \ref{pro-41}, we can obtain the estimate of 
\begin{equation*}
\varphi^{\pm}_1-\left(\varphi_2^{\pm}\circ(\mathfrak{T}_d^{\pm})^{-1}\right),\,\,\psi^{\pm}_1-\left(\psi_2^{\pm}\circ(\mathfrak{T}_d^{\pm})^{-1}\right),\mbox{ and }f_1-f_2
\end{equation*}
for $\mathfrak{T}_d^{\pm}$ given in \eqref{domain-trans}.
To obtain the estimates of 
\begin{equation*}
dS^{\pm}:=S^{\pm}_1-\left(S_2^{\pm}\circ(\mathfrak{T}_d^{\pm})^{-1}\right)\mbox{ and }
d\Lambda^{\pm}:=\frac{\Lambda_1^{\pm}}{r} {\bf e}_{\theta}-\left(\frac{\Lambda_2^{\pm}}{r} {\bf e}_{\theta}\circ(\mathfrak{T}_d^{\pm})^{-1}\right),
\end{equation*}
we use \eqref{S-form}.
Since the entrance conditions of $(S_1^{\pm},\Lambda_1^{\pm})$ and $(S_2^{\pm},\Lambda_2^{\pm})$ are same, we have 
\begin{equation}
\|dS^{\pm}\|_{1,\alpha,\Omega_{f_1,L}^{\pm}}+\|d\Lambda^{\pm}\|_{1,\alpha,\Omega_{f_1,L}^{\pm}}
\le C\sigma \mathfrak{D}^{\pm}
\end{equation}
for $\mathfrak{D}^{\pm}$ given by 
\begin{equation}
\begin{split}
\mathfrak{D}^{\pm}
:=&\,
\|\varphi^{\pm}_1-\left(\varphi_2^{\pm}\circ(\mathfrak{T}_d^{\pm})^{-1}\right)\|_{2,\alpha,\Omega_{f_1,L}^{\pm}}
+\|\psi^{\pm}_1-\left(\psi_2^{\pm}\circ(\mathfrak{T}_d^{\pm})^{-1}\right)\|_{2,\alpha,\Omega_{f_1,L}^{\pm}}\\
&+\|f_1-f_2\|_{2,\alpha,(0,L)}.
\end{split}
\end{equation}
Then, similarly to the proof of Lemma \ref{pro-41}, we can see that there exist small constants $\varepsilon_1^{\star}\in(0,\varepsilon_3]$ depending only on ($\rho_0^{\pm}$, $u_0^{-}$, $p_0$, $\gamma$, $\alpha$)  and  $\sigma_1^{\star}\in(0,\sigma_3]$ depending only on depending only on ($\rho_0^{\pm}$, $u_0^{\pm}$, $p_0$, $\gamma$, $\alpha$) but independent of $L$ so that if $u_0^+\le \varepsilon_1^{\star}$ and $\sigma\le\sigma_1^{\star}$, then the solution is unique. 
%%%%%%%%%%%%%%%%%%%%%%%%%%%%%%%%%%%%%%%%%%%%%%
%%%%%%%%%%%%%%%%%%%%%%%%%%%%%%%%%%%%%%%%%%%%%%
%%
This finishes the proof of Proposition  \ref{Theorem-HD}.\qed

%%%%%%%%%%%%%%%%%%%%%%%%%%%%
%%%%%%%%%%%%%%%%%%%%%%%%%%%%%%%%%%%%%%
%%%%%%%%%%%%%%%%%%%%%%%%%%%%%%%%%%

%\newpage
\subsection{Proof of Lemma \ref{Lemma-fixed}}\label{S-proof-2}
For $\zeta\in\mathbb{R}$, ${\bf a}=(a_1,a_2,a_3)$, ${\bf b}=(b_1,b_2,b_3)\in\mathbb{R}^3$, 
define ${\bf A}^{\pm}=(A_1^{\pm},A_2^{\pm},A_3^{\pm})$ by 
$$A_j^{\pm}(\zeta,{\bf a},{\bf b}):=\varrho^{\pm}(\zeta,{\bf a}+{\bf b})a_j\mbox{ for }j=1,2,3.$$
to rewrite the first equation in \eqref{Fix-prob} as 
\begin{equation}\label{3D-lin-first}
\mbox{div}({\bf A}^{\pm}(S^{\pm},\nabla\varphi^{\pm},{\bf t}^{\pm}))=-\mbox{div}(\varrho^{\pm}(S^{\pm},\nabla\varphi^{\pm}+{\bf t}^{\pm}){\bf t}^{\pm}).
\end{equation}
If we set $\alpha_{ij}^{\pm}$ as  
\begin{equation}\label{def-aij}
\alpha_{ij}^{\pm}:=\partial_{a_j}A_i^{\pm}(S_0^{\pm},\nabla\varphi_0^{\pm},{\bf 0})\mbox{ for }i,j=1,2,3,
\end{equation}
then the matrix $[\alpha_{ij}^{\pm}]_{i,j=1}^3$ is strictly positive and diagonal, and there exists a constant $\nu\in(0,\frac{1}{10}]$ satisfying
\begin{equation}\label{alpha-positive}
\nu\mathbb{I}_3\le[\alpha_{ij}^{\pm}]_{i,j=1}^3\le\frac{1}{\nu}\mathbb{I}_3.
\end{equation}
%%%%%

Set $\phi^{\pm}:=\varphi^{\pm}-\varphi_0^{\pm}$. Then the equation \eqref{3D-lin-first} can be rewritten as 
\begin{equation*}
\mathfrak{L}(\phi^{\pm})=\mbox{div}{\mathfrak F}^{\pm}(S^{\pm}-S_0^{\pm},\nabla\phi^{\pm},{\bf t}^{\pm})
\end{equation*}
for $\mathfrak{L}$ and $\mathfrak{F}^{\pm}=(F_1^{\pm},F_2^{\pm},F_3^{\pm})$ defined by 
\begin{equation}\label{def-FF}
\begin{split}
&\mathfrak{L}(\phi^{\pm}):=\sum_{i=1}^3\alpha_{ii}^{\pm}\partial_{ii}\phi^{\pm},\\
&\begin{split}
	F_i^{\pm}(Q):=&-\int_0^1D_{(\zeta,{\bf b})}A_i^{\pm}({\bf V}_0^{\pm}+tQ)dt\cdot(\zeta,{\bf b})\\
		&-\int_0^1D_{\bf a}A_i^{\pm}({\bf V}_0^{\pm}+tQ)-D_{\bf a}A_i^{\pm}({\bf V}_0^{\pm})dt\cdot{\bf a}-\varrho^{\pm}(\mathfrak{e}({\bf V}_0^{\pm}+Q))b_i
	\end{split}
\end{split}
\end{equation}
with $Q=(\zeta,{\bf a},{\bf b})\in\mathbb{R}\times(\mathbb{R}^3)^2$, ${\bf V}_0^{\pm}=(S_0^{\pm},\nabla\varphi_0^{\pm},{\bf 0})$, $\mathfrak{e}({\bf V}_0^{\pm}+Q)=(S_0^{\pm}+\zeta,\nabla\varphi_0^{\pm}+{\bf a}+{\bf b})$.
From the boundary conditions \eqref{Fix-prob}-\eqref{Fix-bd} for $\varphi^{\pm}$, we prescribe the boundary conditions for $\phi^{\pm}$ as follows:
\begin{equation*}
\begin{split}
&\phi^{\pm}=-\int_{\frac{1}{2}}^r v_{\rm en}^{\pm}(t)dt\mbox{ on }\Gamma_0^{\pm},\quad \partial_r\phi^{+}=0\mbox{ on }\Gamma_{\rm w},\quad\phi^{\pm}=0\mbox{ on }\Gamma_{f_{\ast},L}^{\pm},\\
&\nabla\phi^{+}\cdot{\bf n}^{+}_{f_{\ast}}=-\nabla\varphi_0^{+}\cdot{\bf n}^{+}_{f_{\ast}}=\frac{u_0^{+}f_{\ast}'}{\sqrt{1+|f_{\ast}'|^2}}\mbox{ on }\partial\Omega_{f_{\ast},L}^{+}\cap\mathfrak{C}_{f_{\ast},L},\\
&\phi^-=0\mbox{ on }\partial\Omega_{f_{\ast},L}^{-}\cap\mathfrak{C}_{f_{\ast},L}.
\end{split}
\end{equation*}
%where ${\bf n}_{f_{\ast}}^{\pm}$ is the outward unit normal vector field on $\partial\Omega_{f_{\ast}}^{\pm}\cap\mathfrak{C}_{f_{\ast}}$, i.e.,
%\begin{equation}
%{\bf n}_{f_{\ast}}^{\pm}:=\frac{\pm f_{\ast}'(x_1){\bf e}_x+\mp {\bf e}_r}{\sqrt{1+|f_{\ast}'(x_1)|^2}}.
%\end{equation}

%%%%%%%%%%%%%%%%%%%%%%%%%%%%%%%%%%%%%%%%%%%%%
%%%%%%%%%%%%%%%%%%%%%%%%%%%%%%%%%%%%%%%%%%%%%
%%%%%%%%%%%%%%%%%%%%%%%%%%%%%%%%%%%%%%%%%%%
%%%%%%%%%%%%%%%%%%%%%%%%%%%%%%%%%%%%%%%%%%%%%
%For each $(S_{\ast}^{\pm},\frac{\Lambda_{\ast}^{\pm}}{r}{\bf e}_{\theta},f_{\ast})\in\mathcal{I}(\delta_1)\times\mathcal{I}(\delta_2)$, let us set ${\bf t}_{\ast}^{\pm}:=\nabla\times(\psi^{\pm}{\bf e}_{\theta})+\frac{\Lambda^{\pm}_{\ast}}{r}{\bf e}_{\theta}$ for simplicity of notations.
Now, we solve the boundary value problem for $(\phi^{\pm},{\bf W}^{\pm}:=\psi^{\pm}{\bf e}_{\theta})$:
\begingroup
\begin{align}\label{3D-psi-phi}\allowdisplaybreaks
\begin{aligned}
	&\left\{\begin{aligned}
	\mathfrak{L}(\phi^{\pm})=\mbox{div}{\mathfrak F}^{\pm}(S_{\ast}^{\pm}-S_0^{\pm},\nabla\phi^{\pm},{\bf t}_{\ast}^{\pm})\mbox{ in }\Omega_{f_{\ast},L}^{\pm},\\
	\phi^{\pm}=-\int_{\frac{1}{2}}^r v_{\rm en}^{\pm}(t)dt\mbox{ on }\Gamma_0^{\pm},\quad \partial_r\phi^{+}=0\mbox{ on }\Gamma_{{\rm w},L},\quad\phi^{\pm}=0\mbox{ on }\Gamma_{f_{\ast},L}^{\pm},\\
	\phi^-=0\mbox{ on }\partial\Omega_{f_{\ast},L}^{-}\cap\mathfrak{C}_{f_{\ast},L},\\
	\nabla\phi^{+}\cdot{\bf n}^{+}_{f_{\ast}}=-\nabla\varphi_0^{+}\cdot{\bf n}_{f_{\ast}}^{+}=\frac{u_0^{+}f_{\ast}'}{\sqrt{1+|f_{\ast}'|^2}}\mbox{ on }\partial\Omega_{f_{\ast},L}^{+}\cap\mathfrak{C}_{f_{\ast},L},
	\end{aligned}\right.\\
	&\left\{\begin{aligned}
	-\Delta{\bf W}^{\pm}=G^{\pm}(S_{\ast}^{\pm},\Lambda_{\ast}^{\pm},\partial_rS_{\ast}^{\pm},\partial_r\Lambda_{\ast}^{\pm},\nabla(\phi^{\pm}+\varphi_0^{\pm}),{\bf t}_{\ast}^{\pm}){\bf e}_{\theta}\mbox{ in }\Omega_{f_{\ast},L}^{\pm},\\
	\partial_x{\bf W}^{\pm}={\bf 0}\mbox{ on }\Gamma_0^{\pm},\quad\partial_x{\bf W}^{+}={\bf 0}\mbox{ on }\Gamma_{{\rm w},L},\quad
	\partial_x{\bf W}^{\pm}={\bf 0}\mbox{ on }\Gamma_{f_{\ast},L}^{\pm},\\
	(\nabla\times{\bf W}^-)\cdot{\bm\tau}_{f_{\ast}}=\sqrt{\mathfrak{Y}_{\ast}-\left(\frac{\Lambda_{\ast}}{f_{\ast}}\right)^2}-\nabla\varphi_0^-\cdot{\bm\tau}_{f_{\ast}}\mbox{ on }\partial\Omega_{f_{\ast},L}^{-}\cap\mathfrak{C}_{f_{\ast},L},\\
	(\nabla\times{\bf W}^{+})\cdot{\bf n}_{f_{\ast}}^{+}={\bf 0}\mbox{ on }\partial\Omega_{f_{\ast},L}^{+}\cap\mathfrak{C}_{f_{\ast},L}
	\end{aligned}\right.
\end{aligned}
\end{align}
\endgroup
for $\mathfrak{Y}_{\ast}$ given by
\begin{equation*}
\mathfrak{Y}_{\ast}:=2\left(B_0^--\frac{\gamma}{\gamma-1}(\mathfrak{p}(S_{\ast}^+,{\bf q}_{\ast}^+))^{1-\frac{1}{\gamma}}(S_{\ast}^-)^{\frac{1}{\gamma}}\right).
\end{equation*}
Note that the condition of ${\bf W}^-$ on $\partial\Omega_{f_{\ast},L}^-\cap\mathfrak{C}_{f_{\ast},L}$ is equivalent to 
\begin{equation*}
\nabla{\bf W}^-\cdot{\bf n}_{f_\ast}^--\mu_{f_{\ast}}{\bf W}^-=\mathcal{A}_{\ast}{\bf e}_{\theta}\mbox{ on }\partial\Omega_{f_{\ast},L}^-\cap\mathfrak{C}_{f_{\ast},L}
\end{equation*}
for $\mu_{f_{\ast}}$ and $\mathcal{A}_{\ast}$ defined by 
\begin{equation*}
\mu_{f_{\ast}}:=-\frac{1}{f_{\ast}(x)\sqrt{1+|f'_{\ast}(x)|^2}}
\end{equation*}
and
\begin{equation}\label{bd-A-def}
\mathcal{A}_{\ast}:=-\sqrt{\mathfrak{Y}_{\ast}-\left(\frac{\Lambda_{\ast}}{f_{\ast}}\right)^2}+\nabla\varphi_0^-\cdot{\bm\tau}_{f_{\ast}}.
\end{equation}
Also, the condition of ${\bf W}^+$ on $\partial\Omega_{f_{\ast},L}^+\cap\mathfrak{C}_{f_{\ast},L}$ is equivalent to 
\begin{equation}\label{bc-psi+}
\nabla\psi^+\cdot{\bm\tau}_{f_{\ast}}+\frac{f'
_{\ast}}{f_{\ast}\sqrt{1+|f_{\ast}'|^2}}\psi^+=0\mbox{ on }\partial\Omega_{f_{\ast},L}^+\cap\mathfrak{C}_{f_{\ast},L}.
\end{equation}
Thus, if $\psi^+\equiv 0$ on $\partial\Omega_{f_{\ast},L}^+\cap\mathfrak{C}_{f_{\ast},L}$, then the condition \eqref{bc-psi+} holds.
So we will find ${\bf W}^+$ satisfying 
\begin{equation*}
{\bf W}^+\equiv {\bf 0}\mbox{ on }\partial\Omega_{f_{\ast},L}^+\cap\mathfrak{C}_{f_{\ast},L}.
\end{equation*}

%%%%%%%%%%%%%%%%%%
\begin{lemma}\label{lemma-pphi}
There exists a small constant $\sigma_4>0$ depending only on the data and $(\delta_1,\delta_2)$ but independent of $L$ so that if 
$$\sigma\le\sigma_4,$$
then the boundary value problem \eqref{3D-psi-phi} has a unique axisymmetric solution $(\phi^{\pm},{\bf W}^{\pm})\in [C^{2,\alpha}(\overline{\Omega^{\pm}_{{f_{\ast}},L}})]^4$ that satisfies 
\begin{equation}\label{pphi-est}
\begin{split}
&\|\phi^{+}\|_{2,\alpha,\Omega^{+}_{{f_{\ast}},L}}+\|{\bf W}^+\|_{2,\alpha,\Omega^{+}_{{f_{\ast}},L}}\le C(1+\delta_1^++u_0^+\delta_2)\sigma,\\
&\|\phi^{-}\|_{2,\alpha,\Omega^{-}_{{f_{\ast}},L}}+\|{\bf W}^-\|_{2,\alpha,\Omega^{-}_{{f_{\ast}},L}}\le C(1+\delta_1^-+\delta_1^++u_0^+\delta_2)\sigma
\end{split}
\end{equation}
for a constant $C>0$ depending only on the data but independent of $L$.
Moreover, $\phi^{\pm}$ satisfy
$$\partial_x^2\phi^{\pm}=0\quad\mbox{on }\Gamma_{0,\epsilon}^{\pm}\cup\Gamma_{f_{\ast},L}^{\pm},$$
and ${\bf W}^{\pm}$ have the forms of ${\bf W}^{\pm}=\psi^{\pm}{\bf e}_{\theta}$ for axisymmetric functions $\psi^{\pm}=\psi^{\pm}(x,r)$ satisfying 
\begin{equation}\label{psi-pm-eq}
\left\{\begin{split}
-\left(\partial_{xx}+\frac{1}{r}\partial_r(r\partial_r)-\frac{1}{r^2}\right)\psi^{\pm}=G^{\pm}(S_{\ast}^{\pm},\Lambda_{\ast}^{\pm},\partial_rS_{\ast}^{\pm},\partial_r\Lambda_{\ast}^{\pm},\nabla(\phi^{\pm}+\varphi_0^{\pm}),{\bf t}_{\ast}^{\pm})\mbox{ in }\Omega_{f_{\ast},L}^{\pm},\\
\partial_x\psi^{\pm}=0\mbox{ on }\Gamma_0^{\pm},\quad\partial_x\psi^{+}=0\mbox{ on }\Gamma_{{\rm w},L},\quad
	\partial_x\psi^{\pm}=0\mbox{ on }\Gamma_{f_{\ast},L}^{\pm},\\
%\pm\frac{f_{\ast}'(x)}{r}\partial_r(r\psi^{\pm})\pm\partial_x\psi^{\pm}=0&
-\frac{1}{r}\nabla(r\psi^-)\cdot{\bf n}_{f_{\ast}}^-=\sqrt{\mathfrak{Y}_{\ast}-\left(\frac{\Lambda_{\ast}}{f_{\ast}}\right)^2}-\nabla\varphi_0^-\cdot{\bm\tau}_{f_{\ast}}\mbox{ on }\partial\Omega_{f_{\ast},L}^{-}\cap\mathfrak{C}_{f_\ast,L},\\
\psi^{+}=0\mbox{ on }\partial\Omega_{f_{\ast},L}^{+}\cap\mathfrak{C}_{f_\ast,L},\\
\psi^{-}=\partial_{rr}\psi^-=0\mbox{ on }\Omega_{f_{\ast},L}^-\cap\{r=0\}.
\end{split}\right.
\end{equation}
\end{lemma}
%%%%%%%%%%%%%%%%%%%%%%%%%%%%%%%%%%%%%%%%%%%%%
Hereafter, we regard any estimate constant $C$ to be chosen depending only on the data but independent of $L$ unless specified otherwise.

\begin{remark}
The elliptic estimates are obtained by adapting the regularity theory developed in \cite{bae2019contact3D}.
But, unlike the one-phase setting of \cite{bae2019contact3D}, the pressure continuity condition here couples two unknown Euler phases through the free boundary, and the elliptic variables cannot be estimated independently of this coupling. 
To obtain uniform $C^0$-estimates compatible with the two-phase transmission structure, we introduce comparison functions in the proof below.
These estimates provide the starting point for the subsequent \(C^{2,\alpha}\)-estimates and are essential for closing the coupled iteration scheme.
\end{remark}

%The \(C^0\)-estimates obtained above are a key new point in the present two-phase problem. Unlike the one-phase setting of [4], the pressure continuity condition here couples two unknown Euler phases through the free boundary, and the elliptic variables cannot be estimated independently of this coupling. The comparison functions \(M^\pm\) and \(N^\pm\) are introduced to obtain uniform \(C^0\)-bounds compatible with the two-phase transmission structure. These bounds provide the starting point for the subsequent \(C^{2,\alpha}\)-estimates and are essential for closing the coupled iteration scheme.

%%%%%%%%%%%%%%%%%%%%%%%%%%%%%%%%%%%%%%%%%%%%%
\begin{proof}[Proof of Lemma \ref{lemma-pphi}]
The proof is divided into four steps.

{\bf 1.}
For positive constants $\delta_k^{\pm}$ ($k=3,4$) to be determined later, we define four iteration sets $\mathcal{I}(\delta_k^{\pm})$ $(k=3,4)$ by 
\begingroup
\begin{align*}\label{Ite-set-34}\allowdisplaybreaks %%%%%%Lambda set 수정필요
\begin{aligned}
&\mathcal{I}(\delta_3^{\pm}):=\left\{\phi^{\pm}\in C^{2,\alpha}(\overline{\Omega_{f_\ast,L}^{\pm}})\left|\,\begin{aligned}
											&\phi^{\pm}\mbox{ is axisymmetric,}\\
											%&\phi^{\pm}=-\int_{\frac{1}{2}}^r v_{\rm en}^{\pm}(t)dt\mbox{ on }\Gamma_{0}^{\pm},\\
											&\partial_{x}^2\phi^{\pm}=0\mbox{ on }\Gamma_{0,\epsilon}^{\pm}\cup\Gamma_{f_{\ast},L}^{\pm},\\
											&\partial_{r}\phi^{\pm}=0\mbox{ on }\Gamma_{0,\epsilon}^{\pm}\cup\Gamma_{f_{\ast},L}^{\pm},\\
											%&\phi^{\pm}(x,0)=0\mbox{ for any }x\in[0,L],\\
											&\|\phi^{\pm}\|_{2,\alpha,\Omega^{\pm}_{f_{\ast},L}}\le \delta_3^{\pm}\sigma
											\end{aligned}\right.\right\},\\
&\mathcal{I}(\delta_4^{\pm}):=\left\{{\bf W}^{\pm}=\psi^{\pm}{\bf e}_{\theta}\in \left[C^{2,\alpha}(\overline{\Omega_{f_\ast,L}^{\pm}})\right]^3\left|\,\begin{aligned}
											&\psi^{\pm}\mbox{ is axisymmetric,}\\
											&\partial_x{\bf W}^{\pm}={\bf 0}\mbox{ on }\Gamma_{0}^{\pm}\cup\Gamma_{f_{\ast},L}^{\pm},\\
											&{\bf W}={\bf 0}\mbox{ on }\{r=0\},\\
											&\|{\bf W}^{\pm}\|_{2,\alpha,\Omega^{\pm}_{f_{\ast},L}}\le \delta_4^{\pm}\sigma
											\end{aligned}\right.\right\}.									
\end{aligned}
\end{align*}
\endgroup
%%

%\begin{remark}
%We don't need an iteration set $\mathcal{I}(\delta_4^+)$ since ${\bf W}^+={\bf 0}$. 
%\end{remark}

{\bf 2.}
Fix $(\phi_{\ast}^{\pm},\psi_{\ast}^{\pm}{\bf e}_{\theta})\in\mathcal{I}(\delta_3^{\pm})\times\mathcal{I}(\delta_4^{\pm})$. 
To simplify the notation, let us set
\begin{equation*}
\begin{split}
&{\bf t}_{\ast\ast}^{\pm}:=\nabla\times(\psi_{\ast}^{\pm}{\bf e}_{\theta})+\frac{\Lambda^{\pm}_{\ast}}{r}{\bf e}_{\theta},\\
&G_{\ast}^{\pm}:=G^{\pm}(S_{\ast}^{\pm},\Lambda_{\ast}^{\pm},\partial_rS_{\ast}^{\pm},\partial_r\Lambda_{\ast}^{\pm},\nabla(\phi_{\ast}^{\pm}+\varphi_0^{\pm}),{\bf t}_{\ast\ast}^{\pm})
\end{split}
\end{equation*}
for $G^{\pm}$ given by \eqref{3D-varphi0}. 
By the standard elliptic theory, in each region of $\Omega_{f_{\ast},L}^{\pm}$, the following boundary value problem
\begin{equation}\label{3D-psi-prob}
\left\{\begin{split}
	-\Delta{\bf W}^{\pm}=G^{\pm}_{\ast}{\bf e}_{\theta}\mbox{ in }\Omega_{f_{\ast},L}^{\pm},\\
	\partial_x{\bf W}^{\pm}={\bf 0}\mbox{ on }\Gamma_0^{\pm},\quad{\bf W}^{+}={\bf 0}\mbox{ on }\Gamma_{{\rm w},L},\quad
	\partial_x{\bf W}^{\pm}={\bf 0}\mbox{ on }\Gamma_{f_{\ast},L}^{\pm},\\
	\nabla{\bf W}^-\cdot{\bf n}_{f_\ast}^--\mu_{f_{\ast}}{\bf W}^-=\mathcal{A}_{\ast}{\bf e}_{\theta}
\mbox{ on }\partial\Omega_{f_{\ast},L}^{-}\cap\mathfrak{C}_{f_{\ast},L},\\
	{\bf W}^{+}={\bf 0}\mbox{ on }\partial\Omega_{f_{\ast},L}^{+}\cap\mathfrak{C}_{f_{\ast},L}\\
	\end{split}\right.
\end{equation}
has a unique solution ${\bf W}^{\pm}=(W_1^{\pm},W_2^{\pm},W_3^{\pm})\in [C^{1,\alpha}(\overline{\Omega_{f_{\ast},L}^{\pm}})\cap C^{2,\alpha}({\Omega_{f_{\ast},L}^{\pm}})]^3$.
In \eqref{3D-psi-prob}, $\mathcal{A}_{\ast}$ is given in \eqref{bd-A-def}.
%%%%%%%%%
One can prove that ${\bf W}^{\pm}$ have the forms of 
${\bf W}^{\pm}=\psi^{\pm}(x,r){\bf e}_{\theta}$ for axisymmetric functions $\psi^{\pm}$ satisfying \eqref{psi-pm-eq} by a small adjustment of \cite[Proof of Proposition 3.3]{bae20183}.
So we skip it here. 
We note here that ${\bf W}^+={\bf 0}$ since $G_{\ast}^+=0$ and we don't need the iteration set $\mathcal{I}(\delta_4^+)$.
Thus we set $\delta_4^+=0$.

%%%%%%%%%%%%%%%%%%%%%%%%%%
To obtain a $C^0$ estimate of ${\bf W}^{-}$, define a comparison function $\mathfrak{M}^{-}:\overline{\Omega_{{f_\ast},L}^-}\to\mathbb{R}^+$ by 
\begin{equation*}
\begin{split}
%&\mathfrak{M}^+(x_1,x_2):=-\mathfrak{m}^+x_2^2+4\mathfrak{m}^+,\\
&\mathfrak{M}^{-}(x_1,x_2):=-\mathfrak{m}^-x_2^2+4\mathfrak{m}^-
\end{split}
\end{equation*}
for $\mathfrak{m}^{-}$ given by 
\begin{equation*}
\begin{split}
%&\mathfrak{m}^+:=\|G_{\ast}^+{\bf e}_{\theta}\|_{0,\alpha,\Omega_{{f_{\ast}},L}^+},\\
&\mathfrak{m}^-:=\|G_{\ast}^-{\bf e}_{\theta}\|_{0,\alpha,\Omega_{{f_{\ast}},L}^-}+\|\mathcal{A}_{\ast}{\bf e}_{\theta}\|_{1,\alpha,\partial\Omega_{f_{\ast},L}^-\cap\mathfrak{C}_{f_\ast,L}}.
\end{split}
\end{equation*}
Since $-1\le x_2\le 1$, we have 
\begin{equation*}
0\le 3\mathfrak{m}^{-}\le\mathfrak{M}^{-}\le 4\mathfrak{m}^{-}.
\end{equation*}
Also, since $\frac{1}{4}\le f_{\ast}(x)\le\frac{3}{4}$, $0\le|f'_{\ast}|^2\le1$, and  $-1\le x_2\le 1$, we have
\begin{equation}\label{with}
\begin{split}
\nabla\mathfrak{M}^-\cdot{\bf n}_{f_{\ast}}^-
-\mu_{f_{\ast}}\mathfrak{M}^-
%=&\,\frac{-2\mathfrak{m}^-x_2\cos\theta}{\sqrt{1+|f'_{\ast}(x)|^2}}
%+\frac{-\mathfrak{m}^-x_2^2+4\mathfrak{m}^-}{f_{\ast}(x)\sqrt{1+|f'_{\ast}(x)|^2}}\\
=&\,\frac{\mathfrak{m}^-}{\sqrt{1+|f_{\ast}'(x)|^2}}\left(-2x_2\cos\theta-\frac{x_2^2}{f_{\ast}(x)}+\frac{4}{f_{\ast}(x)}\right)\\
\ge&\,\frac{\mathfrak{m}^-}{\sqrt{1+|f_{\ast}'(x)|^2}}\left(-2+\frac{3}{f_{\ast}(x)}\right)\\
\ge&\,\frac{\mathfrak{m}^-}{\sqrt{1+|\frac{3}{4}|^2}}\left(-2+4\right)\\
\ge&\,\frac{8}{5}\mathfrak{m}^-.
\end{split}
\end{equation}
Set $W_j^{-}:={\bf W}^{-}\cdot{\bf e}_j$ for the unit vector in the positive direction of $x_j$-axis for ${\bf x}=(x_1,x_2,x_3)\in\overline{\Omega^{-}_{f_{\ast},L}}$. 
Then, a direct computation with \eqref{with} yields 
\begin{equation*}
\left\{\begin{split}
\Delta(\mathfrak{M}^{-}\pm{ W}_j^{-})\le 0&\mbox{ in }\Omega^{-}_{{f_{\ast},L}},\\
\partial_x(\mathfrak{M}^{-}\pm{W}_j^{-})=0&\mbox{ on }\Gamma_{0}^{-}\cup\Gamma^{-}_{f_{\ast},L},\\
\mathfrak{M}^{-}\pm{ W}_j^{-}\ge0&\mbox{ on }\Gamma_{{\rm w},L}^-,\\
\nabla(\mathfrak{M}^{-}\pm{W}_j^{-})\cdot{\bf n}_{f_{\ast}}^--\mu_{f_{\ast}}(\mathfrak{M}^-\pm{W}_j^-)\ge 0&\mbox{ on }\partial\Omega_{{f_{\ast},L}}^{-}\cap\mathfrak{C}_{f_{\ast},L}.
\end{split}\right.
\end{equation*}
By the comparison principle and Hopf's lemma, we have %$\psi^{\pm}$ satisfies
\begin{equation*}
-\mathfrak{M}^{-}\le {W}_j^{-}\le \mathfrak{M}^{-}\mbox{ for }j=1,2,3,
\end{equation*}
which provides $C^0$ estimate %of ${\bf W}^{\pm}$ as follows:
\begin{equation*}
\|{\bf W}^{-}\|_{C^0(\overline{\Omega_{f_{\ast},L}^{-}})}\le C\mathfrak{m}^{-}.
\end{equation*}
To obtain $C^{2,\alpha}$ estimates, similarly to  \cite[Proof of Lemma 4.3]{bae2019contact3D}, we apply the even reflection method across the entrance $\Gamma_0^{-}$ and the exit $\Gamma^{-}_{f_{\ast},L}$.
%%%%%%%%%%%%%%%%%%%%%%%%%%%%%
%%%%%%%%%%%%%%%%%%%%%%%%%%%%%
Then we have the estimates 
%${\bf W}^{\pm}$ satisfy
\begin{equation}\label{psi+-est}
\begin{split}
%&\|{\bf W}^{+}\|_{2,\alpha,\Omega_{{f_{\ast},L}}^+}\le C\|G^{+}_{\ast}{\bf e}_{\theta}\|_{0,\alpha,\Omega^{+}_{{f_{\ast},L}}},\\
&\|{\bf W}^{-}\|_{2,\alpha,\Omega_{{f_{\ast},L}}^-}\le C\left(\|G^{-}_{\ast}{\bf e}_{\theta}\|_{0,\alpha,\Omega^{-}_{{f_{\ast}},L}}+\|\mathcal{A}_{\ast}{\bf e}_{\theta}\|_{1,\alpha,\partial\Omega_{{f_{\ast},L}}^{-}\cap\mathfrak{C}_{f_{\ast},L}}\right).
\end{split}
\end{equation}
%Since we have
One can check that there exists a small constant $\sigma_{\natural}>0$ depending only on the data and $(\delta_1,\delta_2,\delta_3^{\pm},\delta_4^{\pm})$ but independent of $L$ so that if 
\begin{equation}\label{sigma-natural}
\sigma\le\sigma_{\natural},
\end{equation}
then we have
\begin{equation}\label{GA-est}
\begin{split}
&\|G^{-}_{\ast}{\bf e}_{\theta}\|_{0,\alpha,\Omega^{\pm}_{{f_{\ast},L}}}
%\le C\left(\|S_{\ast}^{\pm}-S_0^{\pm}\|_{1,\alpha,\Omega^{\pm}_{{f_{\ast},L}}}+\|\frac{\Lambda_{\ast}^{\pm}}{r}\|\right)
\le C\delta_1\sigma,\\
&\|\mathcal{A}_{\ast}{\bf e}_{\theta}\|_{1,\alpha,\Omega_{{f_{\ast},L}}^-\cap\mathfrak{C}_{f_{\ast},L}}\le 
%C\left(\|S_{\ast}^--S_0^-\|_{1,\alpha,\Omega^{\pm}_{{f_{\ast},L}}}+\|f_{\ast}-\frac{1}{2}\|_{2,\alpha,(0,L)}^2+\|p^+-p_0\|\right).
C\left(|\delta_2|^2\sigma+\delta_1+\delta_3^++\delta_4^+\right)\sigma.
\end{split}
\end{equation}
Then it follows from \eqref{psi+-est}-\eqref{GA-est} that 
\begin{equation}\label{W-est-proof}
\begin{split}
%&\|{\bf W}^+\|_{2,\alpha,\Omega^+_{{f_{\ast},L}}}\le C_{\sharp}^+\delta_1^{+}\sigma,\\
&\|{\bf W}^{-}\|_{2,\alpha,\Omega^{-}_{{f_{\ast},L}}}\le C_{\sharp}^-\left(|\delta_2|^2\sigma+\delta_1+\delta_3^++\delta_4^+\right)\sigma\end{split}
\end{equation}
for a  constant $C^{-}_{\sharp}>0$ depending only on the data but independent of $L$.
%%%%%%%%%

%%%%%%%%%
{\bf 3.}
To simplify the notation, let us set 
\begin{equation*}
\mathfrak{F}_{\ast}^{\pm}:={\mathfrak F}^{\pm}(S_{\ast}^{\pm}-S_0^{\pm},\nabla\phi^{\pm}_{\ast},{\bf t}^{\pm}_{\ast})
\end{equation*}
for $\mathfrak{F}$ defined in \eqref{def-FF}.
Now we solve the following boundary value problems for $\phi^{\pm}$;
\begin{equation}\label{3D-phi-pb}
\left\{\begin{split}
	\mathfrak{L}(\phi^{\pm})=\mbox{div}{\mathfrak F}^{\pm}_{\ast}\mbox{ in }\Omega_{f_{\ast},L}^{\pm},\\
	\phi^{\pm}=-\int_{\frac{1}{2}}^r v_{\rm en}^{\pm}(t)dt\mbox{ on }\Gamma_0^{\pm},\quad \partial_r\phi^{+}=0\mbox{ on }\Gamma_{{\rm w},L},\quad\phi^{\pm}=0\mbox{ on }\Gamma_{f_{\ast},L}^{\pm},\\
	\phi^-=0\mbox{ on }\partial\Omega_{f_{\ast},L}^{-}\cap\mathfrak{C}_{f_{\ast},L},\\
	\nabla\phi^{+}\cdot{\bf n}^{+}_{f_{\ast}}=-\nabla\varphi_0^{+}\cdot{\bf n}_{f_{\ast}}^{+}=\frac{u_0^{+}f_{\ast}'}{\sqrt{1+|f_{\ast}'|^2}}\mbox{ on }\partial\Omega_{f_{\ast},L}^{+}\cap\mathfrak{C}_{f_{\ast},L}.
	\end{split}\right.
\end{equation}
% has a unique solution $\phi^{\pm}\in C^{1,\alpha}(\overline{\Omega_{f_{\ast},L}^{\pm}})\cap C^{2,\alpha}({\Omega^{\pm}_{f_{\ast},L}})$. 
 %Moreover, $\phi^{\pm}$ satisfies 
 %$$\partial_{xx}\phi^{\pm}=0\quad\mbox{on }\Gamma_{0,\epsilon}^{\pm}\cup\Gamma_{f_{\ast},L}^{\pm}.$$
%
%
To simply the boundary conditions on $\Gamma_0^{\pm}$, let us set 
 \begin{equation}\label{gen-def}
 \begin{split}
 &\mathfrak{g}^{\pm}(r):=-\int_{\frac{1}{2}\pm\frac{1}{2}}^{r}v_{\rm en}^{\pm}(t)dt,\\%-\int_{\frac{1}{2}\pm\frac{1}{2}}^{\frac{1}{2}}v_{\rm en}^{\pm}(t)dt,\\
& \mathfrak{g}_{\rm en}^{+}(x,r):=\eta(x)\mathfrak{g}^{+}\left(\frac{1-\frac{1}{2}}{1-f_{\ast}(x)}(r-1)+1\right),\\
&\mathfrak{g}_{\rm en}^{-}(x,r):=\eta(x)\mathfrak{g}^{-}\left(\frac{r}{2f_{\ast}(x)}\right)
\end{split}
 \end{equation}
 for a smooth function $\eta$ such that 
 \begin{equation*}
 \eta=1\mbox{ for } x<\frac{L}{10},\,\,\eta=0\mbox{ for }x>\frac{9L}{10},\,\,|\eta'(x)|\le 2, \mbox{ and }\,\,|\eta''(x)|\le 2.
 \end{equation*}
%Set $\phi_{\rm h}^{\pm}:=\phi^{\pm}-\mathfrak{g}_{\rm en}^{\pm}$. 
Then the problems \eqref{3D-phi-pb} can be rewritten as  problems for $\phi_h^{\pm}:=\phi^{\pm}-\mathfrak{g}_{\rm en}^{\pm}$ as follows:
\begin{equation}\label{psi-pb}
	\left\{\begin{split}
	\mathfrak{L}(\phi_h^{\pm})=\mbox{div}{\mathfrak F}^{\pm}_{\ast}+\mathfrak{L}(\mathfrak{g}_{\rm en}^{\pm})=:\mathfrak{H}^{\pm}&\mbox{ in }\Omega_{{f_{\ast},L}}^{\pm},\\
	\phi_h^{\pm}=0\mbox{ on }\Gamma_0^{\pm},\quad \partial_r\phi_h^+=0\mbox{ on }\Gamma_{{\rm w},L},\quad\phi^{\pm}_h=0&\mbox{ on }\Gamma_{f_{\ast},L}^{\pm},\\
	\phi_h^-=0&\mbox{ on }\partial\Omega_{{f_{\ast},L}}^{-}\cap\mathfrak{C}_{f_{\ast},L},\\
\nabla\phi^{+}_h\cdot{\bf n}^{+}_{f_{\ast}}=\frac{u_0^{+}f_{\ast}'}{\sqrt{1+|f_{\ast}'|^2}}=:\mathfrak{h}^+&\mbox{ on }\partial\Omega_{{f_{\ast},L}}^{+}\cap\mathfrak{C}_{f_{\ast},L}.
	\end{split}\right.
\end{equation}
By the standard elliptic theory, each of the problems in \eqref{psi-pb} has a unique solution and  $\phi_h^{\pm}\in C^{1,\alpha}(\overline{\Omega_{{f_{\ast},L}}^{\pm}})\cap C^{2,\alpha}({\Omega^{\pm}_{{f_{\ast}},L}})$.

 %%%%%%%%%%%%%%%%%%%%%%%%%%%%%%%%%%%%%%%%%%%%%%%%%%%%%%%
To obtain $C^0$ estimates of $\phi_h^{\pm}$, define comparison functions $\mathfrak{N}^{\pm}:\overline{\Omega_{{f_\ast},L}^\pm}\to\mathbb{R}^+$ by 
\begin{equation*}
\begin{split}
&\mathfrak{N}^{+}(x_1,x_2,x_3):=\left[\frac{1}{\alpha_{22}^+}\left(-\mathfrak{b}_1\theta^2-\mathfrak{b}_2\sin\left(\frac{\pi}{2}r\right)+\mathfrak{b}_3\right)+1\right]\mathfrak{n}^+,\\
%-\frac{\mathfrak{n}^+}{\alpha_{22}^+}x_2^2+\left(\frac{2}{\alpha_{22}^+}+1\right)\mathfrak{n}^+,\\
&\mathfrak{N}^-(x_1,x_2,x_3):=\left[-\frac{1}{\alpha_{22}^-}x_2^2+\left(\frac{2}{\alpha_{22}^-}+1\right)\right]\mathfrak{n}^-
\end{split}
\end{equation*}
for $\mathfrak{b}_k$ ($k=1,2,3$) and $\mathfrak{n}^{\pm}$ given by 
\begin{equation}\label{def-nn}
\begin{split}
&\mathfrak{b}_1:=\frac{\mathfrak{b}_2\left(\frac{\pi}{2}\right)^2+2}{2},\quad 
\mathfrak{b}_2:=\frac{2\sqrt{1+\left(\frac{3}{4}\right)^2}}{\max|\alpha_{22}^+|\cos\left(\frac{3\pi}{8}\right)\frac{\pi}{2}},\quad \mathfrak{b}_3:=\mathfrak{b}_1(4\pi)^2+\mathfrak{b}_2+1,\\
&\mathfrak{n}^+:=\|\mathfrak{H}^+\|_{1,\alpha,\Omega_{{f_{\ast},L}}^+}+\|\mathfrak{h}^+\|_{1,\alpha,\partial\Omega_{{f_{\ast},L}}^{+}\cap\mathfrak{C}_{f_{\ast},L}},\\
&\mathfrak{n}^-:=\|\mathfrak{H}^-\|_{1,\alpha,\Omega_{{f_{\ast},L}}^-}.
\end{split}
\end{equation}
The comparison argument for \(\mathfrak{N}^+\) is carried out on the cut-open cylindrical coordinate domain \(0\le \theta\le 2\pi\).
%This auxiliary function is used only on the cut-open cylindrical coordinate domain for the purpose of the comparison argument.
Note that %$\mathfrak{N}^{\pm}$ are well-defined since $\Omega\cap\{r=0\}\subset \Omega_{f_{\ast},L}^-$ and 
$0<\nu\le \alpha_{22}^{\pm}\le \frac{1}{\nu}$ for a constant $\nu\in(0,1)$ by \eqref{alpha-positive}.
Also, since $0\le\theta\le 2\pi$, $0\le r\le 1$, and $-1\le x_2\le 1$, we have 
\begin{equation}\label{com-0}
0\le \mathfrak{n}^{+}\le \mathfrak{N}^{+}\le\left(\frac{\mathfrak{b}_3}{\alpha_{22}^+}+1\right)\mathfrak{n}^+\mbox{ and }0\le \mathfrak{n}^-\le\mathfrak{N}^-\le \left(\frac{2}{\alpha_{22}^-}+1\right)\mathfrak{n}^-.
\end{equation}
By a direct computation, we have 
\begin{equation*}
\begin{split}
\Delta\mathfrak{N}^+
&=\partial_{x}^2\mathfrak{N}^++\frac{1}{r}\partial_r\left(r\partial_r\mathfrak{N}^+\right)+\frac{1}{r^2}\partial_{\theta}^2\mathfrak{N}^+\\
&=\frac{1}{\alpha_{22}^+}\left(-\frac{1}{r}\mathfrak{b}_2\cos\left(\frac{\pi}{2}r\right)\frac{\pi}{2}
+\mathfrak{b}_2\sin\left(\frac{\pi}{2}r\right)\left(\frac{\pi}{2}\right)^2
-\frac{2\mathfrak{b}_1}{r^2}\right)\mathfrak{n}^+\\
&\le \frac{1}{\alpha_{22}^+}\left(\mathfrak{b}_2\left(\frac{\pi}{2}\right)^2-\frac{2\mathfrak{b}_1}{r^2}\right)\mathfrak{n}^+\\
&\le \frac{1}{\alpha_{22}^+}\left(-\frac{2}{r^2}\right)\mathfrak{n}^+\le -\frac{2}{\alpha_{22}^+}\mathfrak{n}^+.
\end{split}
\end{equation*}
Since $\partial_{xx}\mathfrak{N}^{\pm}=0$ and $\alpha_{22}^{\pm}=\alpha_{33}^{\pm}$ by the definition of $\alpha_{ij}$ given  in \eqref{def-aij}, we have 
\begin{equation}\label{com-1}
\mathfrak{L}(\mathfrak{N}^+)=\sum_{i=1}^3\alpha_{ii}^+\partial_{ii}\mathfrak{N}^+
=\alpha_{22}^+\Delta\mathfrak{N}^+\le -2\mathfrak{n}^+
\end{equation}
and
\begin{equation}\label{com-ne-1}
\mathfrak{L}(\mathfrak{N}^-)=\sum_{i=1}^3\alpha_{ii}^-\partial_{ii}\mathfrak{N}^-
=\alpha_{22}^-\Delta\mathfrak{N}^-=-2\mathfrak{n}^-.
\end{equation}
Also, by a direct computation, we have
\begin{equation}\label{com-2}
\partial_r\mathfrak{N}^+=-\frac{1}{\alpha_{22}^+}\mathfrak{b}_2\mathfrak{n}^+\cos\left(\frac{\pi}{2}r\right)\frac{\pi}{2}=0\mbox{ for }r=1
\end{equation}
and
\begin{equation}\label{com-3}
\begin{split}
\nabla\mathfrak{N}^+\cdot{\bf n}_{f_{\ast}}^+
&=\frac{-\partial_r\mathfrak{N}^+}{\sqrt{1+|f_{\ast}|^2}}
=\frac{\mathfrak{b}_2\cos\left(\frac{\pi}{2}r\right)\frac{\pi}{2}}{\alpha_{22}^+\sqrt{1+|f_{\ast}|^2}}\mathfrak{n}^+
\ge \frac{\mathfrak{b}_2\cos\left(\frac{3\pi}{8}\right)\frac{\pi}{2}}{\alpha_{22}^+\sqrt{1+\left(\frac{3}{4}\right)^2}}\mathfrak{n}^+
\ge 2\mathfrak{n}^+\mbox{ on }\Omega_{{f_{\ast},L}}^{+}\cap\mathfrak{C}_{f_{\ast},L}
\end{split}
\end{equation}
since $\frac{1}{4}\le r=f_{\ast}(x)\le\frac{3}{4}$ on $\Omega_{{f_{\ast},L}}^{+}\cap\mathfrak{C}_{f_{\ast},L}$.
It follows from \eqref{com-0}-\eqref{com-3} that 
\begin{equation*}
\left\{\begin{split}
\mathfrak{L}(\mathfrak{N}^{+}\pm\phi_h^{+})\le 0&\mbox{ in }\Omega_{{f_{\ast},L}}^{+},\\
\mathfrak{N}^{+}\pm\phi_h^{+}\ge0&\mbox{ on }\Gamma_0^{+}\cup\Gamma_{f_{\ast},L}^{+},\\
\partial_r(\mathfrak{N}^{+}\pm\phi_h^{+})=0&\mbox{ on }\Gamma_{{\rm w},L}^{+},\\
	\nabla(\mathfrak{N}^+\pm\phi_h^+)\cdot{\bf n}_{f_{\ast}}^+\ge 0&\mbox{ on }\partial\Omega_{{f_{\ast},L}}^{+}\cap\mathfrak{C}_{f_{\ast},L},
\end{split}\right.
\end{equation*}
and
\begin{equation*}
\left\{\begin{split}
\mathfrak{L}(\mathfrak{N}^{-}\pm\phi_h^{-})\le0&\mbox{ in }\Omega_{{f_{\ast},L}}^{-},\\
\mathfrak{N}^{-}\pm\phi_h^{-}\ge0&\mbox{ on }\Gamma_0^{-}\cup\Gamma_{f_{\ast},L}^{-},\\
\mathfrak{N}^{-}\pm\phi_h^{-}\ge 0&\mbox{ on }\partial\Omega_{{f_{\ast},L}}^{-}\cap\mathfrak{C}_{f_{\ast},L}.
\end{split}\right.
\end{equation*}
Since $\mathfrak{L}$ is uniformly elliptic, the comparison principle implies that $\phi_h^{\pm}$ satisfy
\begin{equation*}
-\mathfrak{N}^{\pm}\le\phi_h^{\pm}\le\mathfrak{N}^{\pm},
\end{equation*}
from which it follows that 
\begin{equation*}
\|\phi_h^{\pm}\|_{0,\Omega_{f_{\ast},L}^{\pm}}\le C\mathfrak{n}^{\pm}
\end{equation*}
for $\mathfrak{n}^{\pm}$ given in \eqref{def-nn}.

The compatibility conditions for $(f_{\ast},S_{\ast}^{\pm},\Lambda_{\ast}^{\pm},\phi_{\ast}^{\pm}, \psi_{\ast}^{\pm})$ and the definitions of $\mathfrak{g}_{\rm en}^{\pm}$ imply
%$\partial_xS_{\ast}^{\pm}=\partial_x\Lambda_{\ast}^{\pm}=\partial_r\phi_{\ast}^{\pm}=\partial_x^2\phi_{\ast}^{\pm}=\partial_x\psi_{\ast}^{\pm}=0$ on $\Gamma_{0,\epsilon}^{\pm}\cup\Gamma_{f_{\ast},L}^{\pm}$, it holds that $\mbox{div}\mathfrak{F}_{\ast}^{\pm}=0$ on $\Gamma_{0,\epsilon}^{\pm}\cup\Gamma_{f_{\ast},L}^{\pm}$.
%Also, by the definition of $\mathfrak{g}_{\rm en}^{\pm}$, it holds that $\mathfrak{L}(\mathfrak{g}_{\rm en}^{\pm})=0$ on $\Gamma_{0,\epsilon}^{\pm}\cup\Gamma_{f_{\ast},L}^{\pm}$. Thus 
%we have  
%it holds that 
\begin{equation*}
\mathfrak{H}^{\pm}=0\mbox{ on }\Gamma_{0,\epsilon}^{\pm}\cup\Gamma_{f_{\ast},L}^{\pm}\mbox{ and }
\mathfrak{h}^+=0\mbox{ at }x=0,L
\end{equation*}
for $\mathfrak{H}^{\pm}$ and $\mathfrak{h}^{\pm}$ defined in \eqref{psi-pb}.
Similarly to \cite[Proof of Lemma 4.3]{bae2019contact3D}, to obtain $C^{2,\alpha}$ estimates up to the boundary, we can apply an odd reflection method across $\Gamma_{0,\epsilon}^{\pm}\cup\Gamma_{f_{\ast},L}^{\pm}$.
Then, we obtain the estimates 
\begin{equation}\label{est-phi-h}
\begin{split}
&\|\phi_h^{+}\|_{2,\alpha,\Omega_{f_{\ast},L}^{+}}\le C\left(\|\mathfrak{H}^{+}\|_{0,\alpha,\Omega_{{f_{\ast},L}}^{+}}+u_0^+\|f_{\ast}'\|_{1,\alpha,(0,L)}\right),\\
&\|\phi_h^{-}\|_{2,\alpha,\Omega_{{f_{\ast},L}}^{-}}\le C\|\mathfrak{H}^{-}\|_{0,\alpha,\Omega_{f_{\ast},L}^{-}}.
\end{split}
\end{equation}
Since $\phi^{\pm}=\phi_h^{\pm}+\mathfrak{g}_{\rm en}^{\pm}$ by the definitions, the estimates in \eqref{est-phi-h} imply that $\phi^{\pm}$ satisfy
\begin{equation}\label{phi-est-mid}
\begin{split}
&\|\phi^{+}\|_{2,\alpha,\Omega_{f_{\ast},L}^{+}}\le C\left(\|\mathfrak{F}_{\ast}^{+}\|_{1,\alpha,\Omega_{f_{\ast},L}^{+}}+\|\mathfrak{g}_{\rm en}^{+}\|_{2,\alpha,\Omega_{f_{\ast},L}^{+}}+u_0^+\|f_{\ast}'\|_{1,\alpha,(0,L)}\right),\\
&\|\phi^{-}\|_{2,\alpha,\Omega_{f_{\ast},L}^{-}}\le C\left(\|\mathfrak{F}_{\ast}^{-}\|_{1,\alpha,\Omega_{f_{\ast},L}^{-}}+\|\mathfrak{g}_{\rm en}^{-}\|_{2,\alpha,\Omega_{f_{\ast},L}^{-}}\right).
\end{split}
\end{equation}
One can check that there exists a small constant $\sigma_{\sharp}>0$ depending only on the data and $(\delta_1,\delta_2,\delta_3^{\pm},\delta_4^{\pm})$ but independent of $L$ so that if 
\begin{equation}\label{sigma-sharp}
\sigma\le\sigma_{\sharp},
\end{equation}
then we have 
\begin{equation}\label{est-FG}
 \|\mathfrak{F}_{\ast}^{\pm}\|_{1,\alpha,\Omega_{f_{\ast},L}^{\pm}}\le C\left(\delta_1+\delta_4^{\pm}\right)\sigma+C(\delta_3^\pm)^2\sigma^2\mbox{ and }\|\mathfrak{g}_{\rm en}^{\pm}\|_{2,\alpha,\Omega_{f_{\ast},L}^{\pm}}\le C\sigma.
 \end{equation}
 Then it follows from \eqref{phi-est-mid} and \eqref{est-FG} that
\begin{equation}\label{phi-est-proof}
\begin{split}
&\|\phi^{+}\|_{2,\alpha,\Omega^{+}_{f_{\ast},L}}\le C^+_{\flat}(1+u_0^+\delta_2+\delta_4^+)\sigma+C_+^{\flat}(\delta_3^+)^2\sigma^2,\\
&\|\phi^{-}\|_{2,\alpha,\Omega^{-}_{f_{\ast},L}}\le C^-_{\flat}(1+\delta_1+\delta_4^{-})\sigma+C_-^{\flat}(\delta_3^-)^2\sigma^2
\end{split}
\end{equation}
for constants $C^{\pm}_{\flat}>0$ and $C_{\pm}^{\flat}>0$ depending only on the data but independent of $L$.
%%%%

%%%%%%%
For any $\theta\in[0,2\pi)$, define functions $\phi_{\theta}^{\pm}$ by 
\begin{equation*}
\phi_{\theta}^{\pm}({\bf x}):=\phi_h^{\pm}(x_1,x_2\cos\theta-x_3\sin\theta,x_2\sin\theta+x_3\cos\theta).
\end{equation*}
Obviously, it holds that $\phi_{\theta}^-=\phi_{h}^-$ on $\partial\Omega_{f_\ast,L}^-$. 
By the definitions of $\phi_{\theta}^{+}$, we have 
\begin{equation*}
\nabla\phi_{\theta}^+\cdot{\bf n}_{f_{\ast}}^+=\frac{f_{\ast}'\partial_x\phi_{\theta}^+-\partial_r\phi_\theta^+}{\sqrt{1+|f_{\ast}|^2}}=\nabla\phi_h^+\cdot{\bf n}_{f_{\ast}}^+ \mbox{ on }\partial\Omega_{{f_{\ast},L}}^{+}\cap\mathfrak{C}_{f_{\ast},L},
\end{equation*}
from which it holds that  $\phi_{\theta}^+=\phi_h^+$ on $\partial\Omega_{f_\ast,L}^+$.
Since $\alpha_{22}=\alpha_{33}$, $\mathfrak{L}(\phi_\theta^{\pm})=\mathfrak{L}(\phi_h^{\pm})$ hold in $\Omega_{f_{\ast},L}^{\pm}$.
By the uniqueness of a solution to the problem \eqref{psi-pb}, it holds that $\phi_\theta^{\pm}\equiv \phi_h^{\pm}$. Thus $\phi_h^{\pm}$ are axisymmetric.
Since $\mathfrak{g}_{\rm en}^{\pm}$ are axisymmetric, $\phi^{\pm}=\phi_h^{\pm}+\mathfrak{g}_{\rm en}^{\pm}$ are axisymmetric.
%%%
Obviously, by the boundary conditions of $\phi^{\pm}$, it holds that $\partial_r\phi^{\pm}=0$ on $\Gamma_{0,\epsilon}^{\pm}\cup\Gamma_{f_{\ast},L}^{\pm}$.
Since $\partial_r\phi^{\pm}=\mbox{div}\mathfrak{F}_{\ast}^{\pm}=0$ on $\Gamma_{0,\epsilon}^{\pm}\cup\Gamma_{f_{\ast},L}^{\pm}$, we have 
$\partial_x^2\phi^{\pm}=0$ on $\Gamma_{0,\epsilon}^{\pm}\cup\Gamma_{f_{\ast},L}^{\pm}$ by  the first equation in \eqref{3D-phi-pb}.

%%%%%%%%%%%%%%%%%%%%%%%%%%%%%%%%%%%%%%%%%%%%%%%%%%%%%%%
{\bf 4.}
Define  iteration mappings $\mathcal{M}^{\pm}:\mathcal{I}(\delta_3^{\pm})\times\mathcal{I}(\delta_4^{\pm})\to [C^{2,\alpha}(\overline{\Omega^{\pm}_{f_{\ast},L}})]^4$ by 
$$\mathcal{M}^{\pm}:(\phi^{\pm}_{\ast},\psi^{\pm}_{\ast}{\bf e}_{\theta})\mapsto(\phi^{\pm},\psi^{\pm}{\bf e}_{\theta}),$$
where $(\phi^{\pm},\psi^{\pm}{\bf e}_{\theta})$ is the unique axisymmetric solution of \eqref{3D-psi-phi} associated with $(f_{\ast}, S_{\ast}^{\pm},\frac{\Lambda_{\ast}^{\pm}}{r}{\bf e}_{\theta}, \phi_{\ast}^{\pm},\psi_{\ast}^{\pm}{\bf e}_{\theta})$.
%%%%%%%
Recalling the estimates in \eqref{W-est-proof} and \eqref{phi-est-proof}, we have 
\begin{equation}\label{phi-one}
\begin{split}
%&\|\psi^+{\bf e}_{\theta}\|_{2,\alpha,\Omega^+_{{f_{\ast},L}}}\le C_{\sharp}^+\delta_1^{+}\sigma,\\
&\|\psi^{-}{\bf e}_{\theta}\|_{2,\alpha,\Omega^{-}_{{f_{\ast},L}}}\le C_{\sharp}^-\left(|\delta_2|^2\sigma+\delta_1+\delta_3^++\delta_4^+\right)\sigma,
\end{split}
\end{equation}
and
\begin{equation*}
\begin{split}
&\|\phi^{+}\|_{2,\alpha,\Omega^{+}_{f_{\ast},L}}\le C^+_{\flat}(1+u_0^+\delta_2+\delta_1^{+}+\delta_4^{+})\sigma+C_+^{\flat}(\delta_3^+)^2\sigma^2,\\
&\|\phi^{-}\|_{2,\alpha,\Omega^{-}_{f_{\ast},L}}\le C^-_{\flat}(1+\delta_1^{-}+\delta_4^{-})\sigma+C_-^{\flat}(\delta_3^-)^2\sigma^2
\end{split}
\end{equation*}
for positive constants $C_{\sharp}^{-}$, $C^{\pm}_{\flat}$, and $C_{\pm}^{\flat}$ depending only on the data but independent of $L$.
%%%%%%%%
We choose $\sigma_4$, $\delta_3^{\pm}$, and $\delta_4^{-}$ satisfying 
\begin{equation*}
\begin{split}
&\sigma_4\le\min\left\{\frac{1}{4\delta_2}, \frac{1}{2C_+^{\flat}\delta_3^+},\frac{\delta_4^-}{2C_{\sharp}^-\delta_2^2},\frac{1}{2C_-^{\flat}\delta_3^-},\sigma_{\natural},\sigma_{\sharp}\right\}=:\sigma^{\flat},\\
&\delta_3^{+}:=2C_{\flat}^+(1+u_0^+\delta_2),\\
%%%%
%\delta_4^{+}:=C_{\sharp}^+\delta_1^+,\\
%%%%
&\delta_3^{-}:=2C_{\flat}^-(1+\delta_1+\delta_4^-),\,\,\delta_4^{-}:=2C^-_{\sharp}(\delta_1+\delta_3^+)
\end{split}
\end{equation*}
for $\sigma_{\natural}$ and $\sigma_{\sharp}$ given in \eqref{sigma-natural} and \eqref{sigma-sharp}, respectively.
Then we have $(\phi^{\pm},\psi^{\pm}{\bf e}_{\theta})\in\mathcal{I}(\delta_3^{\pm})\times\mathcal{I}(\delta_4^{\pm})$ and
%%%
\begin{equation}\label{final-est-phi-psi}
\begin{split}
&\|\phi^{+}\|_{2,\alpha,\Omega^{+}_{f_{\ast},L}}+\|\psi^+{\bf e}_{\theta}\|_{2,\alpha,\Omega^+_{{f_{\ast},L}}}\le C(1+u_0^+\delta_2)\sigma,\\
&\|\phi^{-}\|_{2,\alpha,\Omega^{-}_{f_{\ast},L}}+\|\psi^{-}{\bf e}_{\theta}\|_{2,\alpha,\Omega^{-}_{{f_{\ast},L}}}\le C(1+u_0^+\delta_2+\delta_1)\sigma.
\end{split}
\end{equation}
%%%%%%%%%%%%%%%%%%%%%%%%%%%%%%%
%

Note that the iteration sets $\mathcal{I}(\delta_3^{\pm})$ and $\mathcal{I}(\delta_4^{\pm})$ are convex and compact subsets of $C^{2,\alpha/2}(\overline{\Omega^{\pm}_{f_{\ast},L}})$ and $\left[C^{2,\alpha/2}(\overline{\Omega^{\pm}_{f_{\ast},L}})\right]^3$, respectively. 
Also, similarly to the mapping $\mathcal{J}_{\rm cd}$ in \eqref{J-cd}, by using the uniqueness of a solution to the problem \eqref{3D-psi-phi}, one can prove that $\mathcal{M}^{\pm}$ are continuous in $\left[C^{2,\alpha/2}(\overline{\Omega^{\pm}_{f_{\ast},L}})\right]^4$.
%%%%%%%%%%%%%%%%%%%%%%%%%%%%
Then  the Schauder fixed point theorem yields that the mappings $\mathcal{M}^{\pm}$ have  fixed points $(\phi^{\pm}_{\sharp},\psi^{\pm}_{\sharp}{\bf e}_{\theta})\in\mathcal{I}(\delta_3^{\pm})\times\mathcal{I}(\delta_4^{\pm})$, respectively.
Obviously, $(\phi^{\pm}_{\sharp},\psi^{\pm}_{\sharp}{\bf e}_{\theta})$ is a solution to the boundary value problem \eqref{3D-psi-phi}, and it satisfies the estimate \eqref{pphi-est} by \eqref{final-est-phi-psi}.

%%%%%%%%%%%%%%%%%
Suppose that there are two solutions $(\phi_1^{\pm},\psi_1^{-}{\bf e}_{\theta})$ and $(\phi_2^{\pm},\psi_2^{-}{\bf e}_{\theta})$.
Like \eqref{psi+-est} and \eqref{phi-one}, there exists a small constant $\sigma_4^{\ast}\in(0,\sigma^{\flat}]$ depending only on the data and ($\delta_1$, $\delta_2)$ but independent of $L$ so that if $\sigma\le\sigma_4^{\ast}$, then we have 
\begin{equation}\label{+est}
\begin{split}
&\|\phi_1^+-\phi_2^+\|_{2,\alpha,\Omega^{+}_{f_{\ast},L}}
\le C^{\flat}\sigma\|{\phi}_1^+-{\phi}_2^+\|_{2,\alpha,\Omega^{+}_{f_{\ast},L}}
\end{split}
\end{equation}
for a constant $C^{\flat}>0$ depending only on the data but independent of $L$
and
\begin{equation}\label{-est}
\begin{split}
&\|\phi_1^--\phi_2^-\|_{2,\alpha,\Omega^{-}_{f_{\ast},L}}
\le C(\delta_1+1)\sigma\|{\phi}_1^--{\phi}_2^-\|_{2,\alpha,\Omega^{-}_{f_{\ast},L}}+ C\|{\psi}_1^-{\bf e}_{\theta}-{\psi}_2^-{\bf e}_{\theta}\|_{2,\alpha,\Omega^{-}_{f_{\ast},L}},\\
&\|\psi_1^-{\bf e}_{\theta}-\psi_2^-{\bf e}_{\theta}\|_{2,\alpha,\Omega^{-}_{f_{\ast},L}}
\le C\delta_1\sigma\left(\|{\phi}_1^--{\phi}_2^-\|_{2,\alpha,\Omega^{-}_{f_{\ast},L}}+ \|{\psi}_1^-{\bf e}_{\theta}-{\psi}_2^-{\bf e}_{\theta}\|_{2,\alpha,\Omega^{-}_{f_{\ast},L}}\right)\\
&\qquad\qquad\qquad\qquad\qquad
+C\|{\phi}_1^+-{\phi}_2^+\|_{2,\alpha,\Omega^{+}_{f_{\ast},L}}.
\end{split}
\end{equation}
If $\sigma$ satisfies
\begin{equation*}
\sigma\le \frac{1}{2C^{\flat}},
\end{equation*}
then \eqref{+est} implies that $\phi_1^+=\phi_2^+$.
This and \eqref{-est} yield that we have 
\begin{equation}
\begin{split}
&\|\phi_1^--\phi_2^-\|_{2,\alpha,\Omega^{-}_{f_{\ast},L}}+\|\psi_1^-{\bf e}_{\theta}-\psi_2^-{\bf e}_{\theta}\|_{2,\alpha,\Omega^{-}_{f_{\ast},L}}\\
&\le C_{\diamond}(\delta_1+1)\sigma\left(\|\phi_1^--\phi_2^-\|_{2,\alpha,\Omega^{-}_{f_{\ast},L}}+\|\psi_1^-{\bf e}_{\theta}-\psi_2^-{\bf e}_{\theta}\|_{2,\alpha,\Omega^{-}_{f_{\ast},L}}\right)
\end{split}
\end{equation}
for a constant $C_{\diamond}>0$ depending only on the data but independent of $L$.
If $\sigma$ satisfies
\begin{equation*}
\sigma\le\frac{1}{2C_{\diamond}(\delta_1+1)},
\end{equation*}
then $(\phi_1^-,\psi_1^-{\bf e}_{\theta})=(\phi_2^-,\psi_2^-{\bf e}_{\theta})$.
Therefore, for 
\begin{equation*}
\sigma\le\sigma_4:=\left\{\sigma_4^{\ast}, \frac{1}{2C^{\flat}},\frac{1}{2C_{\diamond}(\delta_1+1)}\right\},
\end{equation*}
 the solution is unique.  
This finishes the proof of Lemma \ref{lemma-pphi}.
\end{proof}

%%%%%%%%%%%%%%%%%%%%%%%%%%%%%%

Clearly, Lemma \ref{lemma-pphi} implies Lemma \ref{Lemma-fixed}.
The proof of Lemma \ref{Lemma-fixed} is completed.
\qed

%%%%%%%%%%%%%%%%%%%%%%%%%
%%%%%%%%%%%%%%%%%%%%%%%%%%%%%%%%%%%%%%%%%%%%%%%%%%%%%%%
\section{Proof of Theorem \ref{3D-Theorem-HD}}\label{S-6}
%Limiting argument
Let $\varepsilon_1^{\star}$ and $\sigma_1^{\star}$ be from Proposition \ref{Theorem-HD} and suppose that $u_0^+\le\varepsilon_1^{\star}$ and $\sigma\le\sigma_1^{\star}$. 
%According to Proposition \ref{Theorem-HD}, the free boundary problem \eqref{3D-Free-prob-finite}-\eqref{3D-Free-bd-finite} in $\Omega_{L}$ for each $L>0$ has a unique solution that satisfies \eqref{3D-est-Thm-HD-finite}.
For each $n\in\mathbb{N}$, let $(f_n,S_n^{\pm},\frac{\Lambda_n^{\pm}}{r}{\bf e}_{\theta}, \varphi_n^{\pm},\psi_n^{\pm}{\bf e}_{\theta})$ be the solution of \eqref{3D-Free-prob-finite}-\eqref{3D-Free-bd-finite} in $\Omega_{n+100}:=\Omega\cap\{x_1<n+100\}$, and suppose that the solution satisfies the estimate \eqref{3D-est-Thm-HD-finite} given in Proposition \ref{Theorem-HD}. By using the Arzel\'a-Ascoli theorem and a diagonal procedure, we can extract a subsequence 
$\{(f_{n_k},S_{n_k}^{\pm},\frac{\Lambda_{n_k}^{\pm}}{r}{\bf e}_{\theta}, \varphi_{n_k}^{\pm},\psi_{n_k}^{\pm}{\bf e}_{\theta})\}_{n_k\in\mathbb{N}}$ so that the subsequence converges to functions $(f_{\infty},S_{\infty}^{\pm},\frac{ \Lambda_{\infty}^{\pm}}{r}{\bf e}_{\theta}, \varphi_{\infty}^{\pm},\psi_{\infty}^{\pm}{\bf e}_{\theta})$ in the following sense:
For any $L>0$,
\begin{itemize}
\item[(i)] $f_{n_k}$ converges to $f_{\infty}$ in $C^2$ in $[0,L]$;
\item[(ii)] $(S_{n_k}^{\pm}\circ\mathfrak{T}_{n_k}^{\pm},\frac{\Lambda_{n_k}^{\pm}}{r}{\bf e}_{\theta}\circ\mathfrak{T}_{n_k}^{\pm})$ converges to $(S_{\infty}^{\pm},\frac{\Lambda_{\infty}^{\pm}}{r}{\bf e}_{\theta})$ in $C^1$ in $\overline{\Omega^{\pm}_{f_{\infty},L}}$, where $\mathfrak{T}_{n_k}^{\pm}:\overline{\Omega^{\pm}_{f_{\infty},n_k+100}}\to \overline{\Omega^{\pm}_{f_{n_k},n_k+100}}$ are defined by 
\begin{equation*}
\begin{split}
&\mathfrak{T}^{+}_{n_k}(x_1,x_2,x_3):=\left(x_1,\mathcal{G}\frac{x_2}{\sqrt{x_2^2+x_3^2}},\mathcal{G}\frac{x_3}{\sqrt{x_2^2+x_3^2}}\right),\\
&\mathfrak{T}^{-}_{n_k}(x_1,x_2,x_3):=\left(x_1,\frac{f_{n_k}(x_1)}{f_{\infty}(x_1)}x_2,\frac{f_{n_k}(x_1)}{f_{\infty}(x_1)}x_3\right)
\end{split}
\end{equation*}
for
\begin{equation}
\mathcal{G}:=\frac{1-f_{n_k}(x_1)}{1-f_{\infty}(x_1)}\sqrt{x_2^2+x_3^2}+\frac{f_{n_k}(x_1)-f_{\infty}(x_1)}{1-f_{\infty}(x_1)};
\end{equation}
\item[(iii)] $(\varphi_{n_k}^{\pm}\circ\mathfrak{T}_{n_k}^{\pm},(\psi_{n_k}^{\pm}{\bf e}_{\theta})\circ\mathfrak{T}_{n_k}^{\pm})$ converges to $(\varphi_{\infty}^{\pm},\psi_{\infty}^{\pm}{\bf e}_{\theta})$ in $C^2$ in $\overline{\Omega^{\pm}_{f_{\infty},L}}$.
\end{itemize}
%%%
By a change of variables and passing to the limit $n_k\to\infty$, one can prove that $(f_{\infty},S_{\infty}^{\pm},\frac{\Lambda_{\infty}^{\pm}}{r}{\bf e}_{\theta},\varphi_{\infty}^{\pm},\psi_{\infty}^{\pm}{\bf e}_{\theta})$ is a solution to the free boundary problem  \eqref{3D-Free-prob}-\eqref{3D-Free-bd}.
The $C^1$ convergence of $\{(S_{n_k}^{\pm},\frac{\Lambda_{n_k}^{\pm}}{r}{\bf e}_{\theta})\}_{n_k\in\mathbb{N}}$, $C^2$ convergence of $\{(f_{n_k},\varphi_{n_k}^{\pm},\psi_{n_k}^{\pm}{\bf e}_{\theta})\}_{n_k\in\mathbb{N}}$, and the estimate \eqref{3D-est-Thm-HD-finite} imply that $(f_{\infty},S_{\infty}^{\pm},\frac{\Lambda_{\infty}^{\pm}}{r}{\bf e}_{\theta}, \varphi_{\infty}^{\pm},\psi_{\infty}^{\pm}{\bf e}_{\theta})$ satisfies the estimate \eqref{3D-est-Thm-HD}.
The proof of Theorem \ref{3D-Theorem-HD} is completed. \qed
%%%%%%%%%%%%%%%%%%%%%%%%%%%%%%%%%%

%%%%%%%%%%%%%%%%%%%%%%%
\section{ Proof of Theorem \ref{main=Thm-infty}}\label{S-7}
%\subsection{}
\subsection{Proof of (a): Existence}\label{sub-ex}
Let $\varepsilon_1$ and $\sigma_1$ be from Theorem \ref{3D-Theorem-HD}, and suppose that $u_0^+\le\varepsilon_1$ and $\sigma\le \sigma_1$.
Suppose that $(f,S^{\pm},\Lambda^{\pm},\varphi^{\pm},\psi^{\pm})$ is an axisymmetric solution to  \eqref{3D-Free-prob}-\eqref{3D-Free-bd} and satisfies the estimate \eqref{3D-est-Thm-HD}.
For such a solution, we define ${\bf u}^{\pm},$ $\rho^{\pm},$ and $p^{\pm}$ by
\begin{equation}\label{real-sol}
{\bf u}^{\pm}:=\nabla\varphi^{\pm}+\nabla\times(\psi^{\pm}{\bf e}_{\theta})+\frac{\Lambda^{\pm}}{r}{\bf e}_{\theta},\quad
\rho^{\pm}:=\varrho^{\pm}(S^{\pm},{\bf u}^{\pm}),\quad
p^{\pm}:=S^{\pm}(\varrho^{\pm})^{\gamma}(S^{\pm},{\bf u}^{\pm})
\end{equation}
for $\varrho^{\pm}$ given in \eqref{3D-varphi0}.
Then, one can choose small constants $\varepsilon_0>0$ depending only on ($\rho_0^{\pm}$, $u_0^{-}$, $p_0$, $\gamma$, $\alpha$) and $\sigma_0>0$ depending only on ($\rho_0^{\pm}$, $u_0^{\pm}$, $p_0$, $\gamma$, $\alpha$) so that  if $u_0^+\le\varepsilon_0$ and $\sigma\le\sigma_0$, then $(f, \rho^{\pm},{\bf u}^{\pm},p^{\pm})$ is a solution to Problem \ref{P-1-infty} and the estimate \eqref{3D-est-Thm-HD} implies the estimate \eqref{u-rho-p-est-infty}.

\subsection{Proof of (b): Asymptotic behavior}\label{sub-asym}
To observe the asymptotic behavior of the solution, we follow  the computations in \cite[Proof of Theorem 2.1(b)]{bae2019contact3D}.
The main difference from \cite{bae2019contact3D} is the computation on the contact discontinuity because the pressure is not a constant on the contact discontinuity here. 
However, the transport structure of the entropy and the angular momentum density still allows us to derive the key estimates on the contact discontinuity in essentially the same way as in \cite{bae2019contact3D}.
In the following, we provide the details.

Define functions $\mathfrak{h}^{\pm}$ by 
\begin{equation*}
\begin{split}
&\mathfrak{h}^+(x,r):=\int_1^r t\rho{\bf u}^+\cdot{\bf e}_x(x,t)dt\mbox{ for }f(x)<r<1,\\
&\mathfrak{h}^{-}(x,r):=\int_0^r t\rho {\bf u}^-\cdot{\bf e}_x(x,t)dt\mbox{ for }0<r<f(x).
\end{split}
\end{equation*}
By the continuity equation, we have 
\begin{equation}\label{def-stream}
\partial_x\mathfrak{h}^{\pm}=-r\rho^{\pm}{\bf u}^{\pm}\cdot{\bf e}_r\mbox{ and }
\partial_r\mathfrak{h}^{\pm}=r\rho^{\pm}{\bf u}^{\pm}\cdot{\bf e}_x.
\end{equation}
Since the entropy and the angular momentum density are transported along streamlines, we can represent $S^{\pm}$ and $\Lambda^{\pm}$ as 
\begin{equation}\label{S-conserv}
S^{\pm}(x,r)=\tilde{S}^{\pm}(\mathfrak{h}^{\pm}(x,r))\mbox{ and }\Lambda^{\pm}(x,r)=\tilde{\Lambda}^{\pm}(\mathfrak{h}^{\pm}(x,r)).
\end{equation}
For notational simplicity, we drop $\tilde{}$.
Set
\begin{equation*}
\mathfrak{S}^+(\mathfrak{h}^+):=\frac{\gamma}{\gamma-1}S^+(\mathfrak{h}^+)=\frac{\gamma}{\gamma-1}S_0^+,\quad
\mathfrak{S}^-(\mathfrak{h}^-):=\frac{\gamma}{\gamma-1}S^-(\mathfrak{h}^-).
\end{equation*}
Then the definition of the Bernoulli invariant yields that
\begin{equation}\label{Ber-diff}
\begin{split}
&B_0^{+}r^2(\rho^{+})^2=\frac{1}{2}|\nabla\mathfrak{h}^{+}|^2+r^2\mathfrak{S}^+(\mathfrak{h}^+)(\rho^{+})^{\gamma+1},\\
&B_0^{-}r^2(\rho^{-})^2=\frac{1}{2}\left(|\nabla\mathfrak{h}^{-}|^2+(\Lambda^{-}(\mathfrak{h}^{-}))^2(\rho^{-})^2\right)+r^2\mathfrak{S}^{-}(\mathfrak{h}^{-})(\rho^{-})^{\gamma+1},
\end{split}
\end{equation}
where $\nabla=(\partial_x,\partial_r)$.
Differentiating \eqref{Ber-diff} with respect to $x$ and $r$ gives
\begin{equation}\label{rho-rho}
\begin{split}
\partial_x\rho^{+}=&\,\frac{(\partial_x\mathfrak{h}^+)(\partial_{xx}\mathfrak{h}^+)+(\partial_r\mathfrak{h}^+)(\partial_{xr}\mathfrak{h}^+)}{2r^2\rho^+B_0^+-(\gamma+1)\mathfrak{S}^+r^2(\rho^+)^{\gamma}},\\
\partial_r\rho^+=&\,\frac{(\partial_x\mathfrak{h}^+)(\partial_{xr}\mathfrak{h}^+)+(\partial_r\mathfrak{h}^+)(\partial_{rr}\mathfrak{h}^+)+2r\mathfrak{S}^+(\rho^+)^{\gamma+1}}{2r^2\rho^+B_0^+-(\gamma+1)\mathfrak{S}^+r^2(\rho^+)^{\gamma}},\\
\partial_x\rho^-=&\,\frac{(\partial_x\mathfrak{h}^-)(\partial_{xx}\mathfrak{h}^-)+(\partial_r\mathfrak{h}^-)(\partial_{xr}\mathfrak{h}^-)+\Lambda^-(\Lambda^-)'(\partial_x\mathfrak{h}^-)(\rho^-)^2+r^2(\mathfrak{S}^-)'(\partial_x\mathfrak{h}^-)(\rho^-)^{\gamma+1}}{2r^2\rho^-B_0^--(\gamma+1)\mathfrak{S}^-r^2(\rho^-)^{\gamma}-(\Lambda^-)^2\rho^-},\\
\partial_r\rho^-=&\,\frac{(\partial_x\mathfrak{h}^-)(\partial_{xr}\mathfrak{h}^-)+(\partial_r\mathfrak{h}^-)(\partial_{rr}\mathfrak{h}^-)+\Lambda^-(\Lambda^-)'(\partial_r\mathfrak{h}^-)(\rho^-)^2+r^2(\mathfrak{S}^-)'(\partial_r\mathfrak{h}^-)(\rho^-)^{\gamma+1}}{2r^2\rho^-B_0^--(\gamma+1)\mathfrak{S}^-r^2(\rho^-)^{\gamma}-(\Lambda^-)^2\rho^-}\\
&+\frac{2r\mathfrak{S}^-(\rho^-)^{\gamma+1}-2rB_0^-(\rho^-)^2}{2r^2\rho^-B_0^--(\gamma+1)\mathfrak{S}^-r^2(\rho^-)^{\gamma}-(\Lambda^-)^2\rho^-}.
\end{split}
\end{equation}
Using \eqref{def-stream}-\eqref{rho-rho}, the equation 
\begin{equation*}
\rho^{\pm}({\bf u}^{\pm}\cdot{\bf e}_x\partial_x+{\bf u}^{\pm}\cdot{\bf e}_r\partial_r){\bf u}^{\pm}\cdot{\bf e}_r-\frac{\rho^{\pm}({\bf u}^{\pm}\cdot{\bf e}_{\theta})^2}{r}+\partial_r p^{\pm}=0
\end{equation*}
can be rewritten as 
\begin{equation}\label{h-eq-eq}
\begin{split}
&\nabla\cdot\left(\frac{\nabla\mathfrak{h}^{+}}{r\rho^{+}}\right)=0,\\
&\nabla\cdot\left(\frac{\nabla\mathfrak{h}^{-}}{r\rho^{-}}\right)=-\frac{r}{\gamma}(\mathfrak{S}^-)'(\rho^{-})^{\gamma}-\frac{\Lambda^-(\Lambda^-)'\rho^{-}}{r}.
\end{split}
\end{equation}
Set
\begin{equation*}
\omega^{\pm}:=\partial_x\mathfrak{h}^{\pm}
\end{equation*}
and differentiate \eqref{h-eq-eq} with respect to $x$ to get the following equations for $\omega^{\pm}$:
\begin{equation}\label{eq-omega-pm}
\sum_{i,j=1}^2\partial_i\left(a_{ij}^{\pm}\partial_j\omega^{\pm}+a_i^{\pm}\omega^{\pm}\right)=\partial_x\left(-\frac{r}{\gamma}(\mathfrak{S}^{\pm})'(\rho^{\pm})^{\gamma}-\frac{\Lambda^{\pm}(\Lambda^{\pm})'\rho^{\pm}}{r}\right),\quad (\partial_1,\partial_2):=(\partial_x,\partial_r)
\end{equation}
for  $a_{ij}^{\pm}$ and $a_i^{\pm}$ given by 
\begin{equation*}
\begin{split}
&a_{ij}^{\pm}:=\frac{1}{r\rho^{\pm}}\delta_{ij}-\frac{(\partial_i\mathfrak{h}^{\pm})(\partial_j\mathfrak{h}^{\pm})}{r(\rho^{\pm})^2A_1^{\pm}},\quad A_1^{\pm}:=2r^2\rho^{\pm}B_0^{\pm}-(\gamma+1)\mathfrak{S}^{\pm}r^2(\rho^{\pm})^{\gamma}-(\Lambda^{\pm})^2\rho^{\pm},\\
&a_i^{\pm}:=-\frac{(\partial_i \mathfrak{h}^{\pm})A_2^{\pm}}{r(\rho^{\pm})^2A_1^{\pm}},\quad A_2^{\pm}:=\Lambda^{\pm}(\Lambda^{\pm})'(\rho^{\pm})^2+r^2(\mathfrak{S}^{\pm})'(\rho^{\pm})^{\gamma+1}.
\end{split}
\end{equation*}
By the boundary conditions and the compatibility condition, $\omega^{\pm}$ satisfy
\begin{equation}
\omega^{\pm}=r\rho^{\pm}v_{\rm en}^{\pm}\mbox{ on }\Gamma_{\rm en}^{\pm},\quad
\omega^+=0\mbox{ on }\{r=1\},\quad \omega^-=0\mbox{ on }\{r=0\}.
\end{equation} 
Now we compute conormal boundary conditions for \eqref{eq-omega-pm} on the contact discontinuity $\{r=f(x)\}$. 
From the definition of the Bernoulli invariant, we have 
\begin{equation}\label{diff-tan}
B_0^{\pm}(\rho^{\pm})^2=\frac{1}{2}\left(|\partial_x\mathfrak{h}^{\pm}|^2+|\partial_r\mathfrak{h}^{\pm}|^2\right)+\frac{1}{2}\left(\frac{\Lambda^{\pm}}{r}\right)^2(\rho^{\pm})^2+\mathfrak{S}^{\pm}(\rho^{\pm})^{\gamma+1}.
\end{equation}
Differentiating the equation \eqref{diff-tan} in the tangential direction along $\{r=f(x)\}$ gives 
\begin{equation}\label{tan-tan}
\begin{split}
&2B_0^{\pm}\rho^{\pm}(\partial_x\rho^{\pm})-(\partial_x\mathfrak{h}^{\pm})(\partial_{xx}\mathfrak{h}^{\pm})-(\partial_r\mathfrak{h}^{\pm})(\partial_{rx}\mathfrak{h}^{\pm})\\
&-\frac{\Lambda^{\pm}(\partial_x\Lambda^{\pm})}{r}(\rho^{\pm})^2-\left(\frac{\Lambda^{\pm}}{r}\right)^2\rho^{\pm}(\partial_x\rho^{\pm})-(\partial_x\mathfrak{S}^{\pm})(\rho^{\pm})^{\gamma+1}-\mathfrak{S}^{\pm}(\gamma+1)(\rho^{\pm})^{\gamma}(\partial_x\rho^{\pm})\\
&+2B_0^{\pm}\rho^{\pm}(\partial_r\rho^{\pm})f'-(\partial_x\mathfrak{h}^{\pm})(\partial_{xr}\mathfrak{h}^{\pm})f'-(\partial_r\mathfrak{h}^{\pm})(\partial_{rr}\mathfrak{h}^{\pm})f'\\
&-\frac{\Lambda^{\pm}(\partial_r\Lambda^{\pm})}{r}(\rho^{\pm})^2f'-\left(\frac{\Lambda^{\pm}}{r}\right)^2\rho^{\pm}(\partial_r\rho^{\pm})f'-(\partial_r\mathfrak{S}^{\pm})(\rho^{\pm})^{\gamma+1}f'-\mathfrak{S}^{\pm}(\gamma+1)(\rho^{\pm})^{\gamma}(\partial_r\rho^{\pm})f'=0.
\end{split}
\end{equation}
Since the entropy and the angular momentum density are transported along streamlines, and since the contact discontinuity is a streamline, and $f'=-\frac{\partial_x\mathfrak{h}^{\pm}}{\partial_r\mathfrak{h}^{\pm}}$, the equation \eqref{tan-tan} can be rewritten as 
%\begin{equation}\label{tan-tan-re}
%\begin{split}
%&2B_0^{\pm}\rho^{\pm}(\partial_x\rho^{\pm})(\partial_r\mathfrak{h}^{\pm})-(\partial_x\mathfrak{h}^{\pm})(\partial_{xx}\mathfrak{h}^{\pm})(\partial_r\mathfrak{h}^{\pm})-(\partial_r\mathfrak{h}^{\pm})(\partial_{rx}\mathfrak{h}^{\pm})(\partial_r\mathfrak{h}^{\pm})\\
%&-\left(\frac{\Lambda^{\pm}}{r}\right)^2\rho^{\pm}(\partial_x\rho^{\pm})(\partial_r\mathfrak{h}^{\pm})-\mathfrak{S}^{\pm}(\gamma+1)(\rho^{\pm})^{\gamma}(\partial_x\rho^{\pm})(\partial_r\mathfrak{h}^{\pm})\\
%&+2B_0^{\pm}\rho^{\pm}(\partial_r\rho^{\pm})(-\partial_x\mathfrak{h}^{\pm})-(\partial_x\mathfrak{h}^{\pm})(\partial_{xr}\mathfrak{h}^{\pm})(-\partial_x\mathfrak{h}^{\pm})-(\partial_r\mathfrak{h}^{\pm})(\partial_{rr}\mathfrak{h}^{\pm})(-\partial_x\mathfrak{h}^{\pm})\\
%&-\left(\frac{\Lambda^{\pm}}{r}\right)^2\rho^{\pm}(\partial_r\rho^{\pm})(-\partial_x\mathfrak{h}^{\pm})-\mathfrak{S}^{\pm}(\gamma+1)(\rho^{\pm})^{\gamma}(\partial_r\rho^{\pm})(-\partial_x\mathfrak{h}^{\pm})=0.
%\end{split}
%\end{equation}
%
\begin{equation}\label{xr-point}
\partial_{xr}\mathfrak{h}^{\pm}=\frac{A_3(\partial_x\mathfrak{h}^{\pm})+A_4(\partial_x\rho^{\pm})}{(\partial_r\mathfrak{h}^{\pm})^2-(\partial_x\mathfrak{h}^{\pm})^2}
\end{equation}
for $A_3$ and $A_4$ given by 
\begin{equation*}
\begin{split}
A_3:=&\,-(\partial_{xx}\mathfrak{h}^{\pm})(\partial_r\mathfrak{h}^{\pm})-2B_0^{\pm}\rho^{\pm}(\partial_r\rho^{\pm})+(\partial_{xr}\mathfrak{h}^{\pm})(-\partial_x\mathfrak{h}^{\pm})+(\partial_r\mathfrak{h}^{\pm})(\partial_{rr}\mathfrak{h}^{\pm})\\
&+\left(\frac{\Lambda^{\pm}}{r}\right)^2\rho^{\pm}(\partial_r\rho^{\pm})+\mathfrak{S}^{\pm}(\gamma+1)(\rho^{\pm})^{\gamma}(\partial_r\rho^{\pm}),\\
A_4:=&\,\frac{1}{r^2}A_1(\partial_r\mathfrak{h}^{\pm}).%\left(2B_0^{\pm}\rho^{\pm}-\left(\frac{\Lambda^{\pm}}{r}\right)^2\rho^{\pm}-\mathfrak{S}^{\pm}(\gamma+1)(\rho^{\pm})^{\gamma}\right).
\end{split}
\end{equation*}
Since 
\begin{equation*}
\partial_x\rho^{\pm}=\,\frac{(\partial_x\mathfrak{h}^{\pm})(\partial_{xx}\mathfrak{h}^{\pm})+(\partial_r\mathfrak{h}^{\pm})(\partial_{xr}\mathfrak{h}^{\pm})+\Lambda^{\pm}(\Lambda^{\pm})'(\partial_x\mathfrak{h}^{\pm})(\rho^{\pm})^2+r^2(\mathfrak{S}^{\pm})'(\partial_x\mathfrak{h}^{\pm})(\rho^{\pm})^{\gamma+1}}{A_1},
\end{equation*}
we have 
\begin{equation}\label{A4}
A_4(\partial_x\rho^{\pm})=\frac{\partial_r\mathfrak{h}^{\pm}}{r^2}\left(A_5(\partial_x\mathfrak{h}^{\pm})+(\partial_r\mathfrak{h}^{\pm})(\partial_{xr}\mathfrak{h}^{\pm})\right)
\end{equation}
for $A_5$ given by 
\begin{equation*}
A_5:=(\partial_{xx}\mathfrak{h}^{\pm})+\Lambda^{\pm}(\Lambda^{\pm})'(\rho^{\pm})^2+r^2(\mathfrak{S}^{\pm})'(\rho^{\pm})^{\gamma+1}.
\end{equation*}
By \eqref{xr-point}-\eqref{A4}, we have 
\begin{equation*}
\begin{split}
\partial_{xr}\mathfrak{h}^{\pm}=&\,\frac{(\partial_x\mathfrak{h}^{\pm})\left(A_3+\frac{1}{r^2}A_5(\partial_r\mathfrak{h}^{\pm})\right)}{(\partial_r\mathfrak{h}^{\pm})^2-(\partial_x\mathfrak{h}^{\pm})^2}\left(1-\frac{(\partial_r\mathfrak{h}^{\pm})^2}{r^2\left((\partial_r\mathfrak{h}^{\pm})^2-(\partial_x\mathfrak{h}^{\pm})^2\right)}\right)^{-1}\\
=:&\,\omega A_6
\end{split}
\end{equation*}
which yields that 
\begin{equation}\label{ess-xr}
\begin{split}
a_{21}^{\pm}\partial_x\omega^{\pm}+a_{22}^{\pm}\partial_r\omega^{\pm}
=&\,-\frac{(\partial_x\mathfrak{h}^{\pm})(\partial_r\mathfrak{h}^{\pm})}{r(\rho^{\pm})^2A_1^{\pm}}\partial_{xx}\mathfrak{h}^{\pm}
+\left(\frac{1}{r\rho^{\pm}}-\frac{(\partial_r\mathfrak{h}^{\pm})^2}{r(\rho^{\pm})^2A_1^{\pm}}\right)\partial_{xr}\mathfrak{h}^{\pm}\\
=&\, \omega A_7
\end{split}
\end{equation}
for $A_7$ given by 
\begin{equation*}
A_7:=-\frac{(\partial_r\mathfrak{h}^{\pm})(\partial_{xx}\mathfrak{h}^{\pm})}{r(\rho^{\pm})^2A_1^{\pm}}+\left(\frac{1}{r\rho^{\pm}}-\frac{(\partial_r\mathfrak{h}^{\pm})^2}{r(\rho^{\pm})^2A_1^{\pm}}\right)A_6.
\end{equation*}
%%%%%%
By \eqref{ess-xr}, we have 
\begin{equation}\label{co-no-bd}
\left(a_{11}^{\pm}\partial_x\omega^{\pm}+a_{12}^{\pm}\partial_r\omega^{\pm},a_{21}^{\pm}\partial_x\omega^{\pm}+a_{22}^{\pm}\partial_r\omega^{\pm}\right)\cdot{\bf n}^{\pm}=\frac{\mathfrak{w}\omega}{\sqrt{1+|f'|^2}}
\end{equation}
for $\mathfrak{w}$ given by 
\begin{equation*}
\mathfrak{w}:=-\frac{a_{11}^{\pm}\partial_x\omega^{\pm}+a_{12}^{\pm}\partial_r\omega^{\pm}}{\partial_r\mathfrak{h}^{\pm}}+A_7.
\end{equation*}
The rest of the calculation is then almost the same as  \cite[Proof of Theorem 2.1(b)]{bae2019contact3D}.
Thus we only state the outline of the rest. %More more details, one can refer to  \cite[Proof of Theorem 2.1(b)]{bae2019contact3D}.
Using \eqref{co-no-bd} and following the energy estimates in \cite[Proof of Theorem 2.1(b)]{bae2019contact3D}, we can show that 
\begin{equation*}
\int_0^{\infty}\int_{f(x)}^1|\nabla\omega^+|^2drdx\le C\mbox{ and }
\int_0^{\infty}\int_0^{f(x)}|\nabla\omega^-|^2drdx\le C.
\end{equation*}
Combining the above energy estimate with the uniform $C^{1,\alpha}$ estimates for $\omega^{\pm}$ yields
\begin{equation}\label{nabla-omega}
\begin{split}
\|\nabla\omega^{\pm}(x,\cdot)\|_{C^0(\overline{\Omega_{f}^{\pm}\cap\{x>L\}})}\to 0\mbox{ as }L\to \infty.
\end{split}
\end{equation}
Since $\omega^-=0$ on $\{r=0\}$ and $\omega^+=0$ on $\{r=1\}$, we have 
\begin{equation*}
\|\omega^{\pm}(x,\cdot)\|_{C^0(\overline{\Omega_{f}^{\pm}\cap\{x>L\}})}\to 0\mbox{ as }L\to \infty.
\end{equation*}
Since $\rho^{\pm}>\frac{\rho_0^{\pm}}{2}$ for a sufficiently small $\sigma$ and $\omega=-r\rho^{\pm}{\bf u}^{\pm}\cdot{\bf e}_r$, we have 
 \begin{equation}\label{ur-c0}
\begin{split}
\|r{\bf u}^{\pm}\cdot{\bf e}_r(x,\cdot)\|_{C^0(\overline{\Omega_{f}^{\pm}\cap\{x>L\}})}\to 0\mbox{ as }L\to \infty.
\end{split}
\end{equation}
By \eqref{nabla-omega} and \eqref{ur-c0}, we have
 \begin{equation*}\label{ur-c1}
\begin{split}
\|r{\bf u}^{\pm}\cdot{\bf e}_r(x,\cdot)\|_{C^1(\overline{\Omega_{f}^{\pm}\cap\{x>L\}})}\to 0\mbox{ as }L\to \infty,
\end{split}
\end{equation*}
from which 
\begin{equation*}
\begin{split}
&\|f'(x)\|_{C^1(\{x\ge L\})}\to0,\\
&\|{\bf u}^{\pm}\cdot{\bf e}_r(x,\cdot)\|_{C^1(\overline{\Omega_{f}^{\pm}\cap\{x>L\}})}\to 0\mbox{ as }L\to \infty.
\end{split}
\end{equation*}
%using ${\bf u}^{\pm}\cdot{\bf e}_r(x,0)=0$ and the axisymmetric regularity.
Near the axis $\{r=0\}$, this follows from the axisymmetric regularity condition ${\bf u}^-\cdot{\bf e}_r(x,0)=0$ and the standard relation between $r{\bf u}^{\pm}\cdot{\bf e}_r$ and ${\bf u}^{\pm}\cdot{\bf e}_r$ away from the axis the conclusion is immediate.
From this and the momentum equation, we have 
\begin{equation*}
\|\partial_rp^{\pm}(x,\cdot)-\frac{\rho^{\pm}({\bf u}^{\pm}\cdot{\bf e}_{\theta})^2}{r}(x,\cdot)\|_{C^0(\overline{\Omega_{f}^{\pm}\cap\{x>L\}})}\to 0\mbox{ as }L\to \infty.
\end{equation*}
This finishes the proof of Theorem \ref{main=Thm-infty}.
\qed
%%%%%%%%%%%%%%%%%%%%%%%%%%%%%%%%%%%%%%%%%%%%%%%%%%%%%%%%%%%%%%%%%%%%%%%%%%%%%%%%%%%%%%%%%%%%%%%%%%%%%%%%%%%%%%%%%%%%%%%%%%%%%%%%

%%%%%%%%%%%%%%%%%%%%%%%%%%%%%%%%%%%%%%%%%%%%%%%%%%%%%%%
\vspace{.25in}
\noindent
{\bf Acknowledgments.}\,\,
%The author expresses gratitude for the encouragement and support of Myoungjean Bae, Gui-Qiang G. Chen, Mikhail Feldman, and Joonhyun La.\\
%The research of Hyangdong Park was supported in part by the POSCO Science Fellowship of POSCO TJ Park Foundation.\\
This research was supported in part by the Ministry of Education of the Republic of Korea and the National Research Foundation of Korea (NRF-RS-2025-00553734).\\

%%%%%%%%%%%%%%%%%%%%%%%%%%%%%%%%%%%%
%%%%%%%%%%%%%%%%%%%%%%%%%%%%%%%%%%%%

%\bigskip

\noindent
{\bf Conflict of interests.}\,\,
There is no conflict of interest.\\

\noindent
{\bf Data availability statement.}\,\,
Data sharing is not applicable to this article as no datasets were generated or analyzed during the current study.

\bigskip
\bibliographystyle{siam}
\bibliography{References}

\end{document}